\documentclass[11pt]{amsart}
\usepackage{amsmath}
\usepackage{amssymb}
\usepackage{amsfonts}
\usepackage{amsthm}
\usepackage{hyperref}
\usepackage{latexsym}
\usepackage{graphics}
\usepackage{graphicx}
\usepackage[all]{xy}
\usepackage{verbatim}

\newtheorem{Definition}{Definition}
\newtheorem{Lemma}{Lemma}

\newtheorem{Theorem}{Theorem}
\newtheorem{Proposition}{Proposition}

\theoremstyle{definition}
\newtheorem*{Ack}{Acknowledgement}

\DeclareMathOperator{\Ham}{Ham}
\DeclareMathOperator{\Inn}{Inn}
\DeclareMathOperator{\Aut}{Aut}
\DeclareMathOperator{\End}{End}
\DeclareMathOperator{\Hom}{Hom}
\DeclareMathOperator{\geo}{geo}
\DeclareMathOperator{\alg}{alg}
\DeclareMathOperator{\nor}{nor}
\DeclareMathOperator{\can}{can}
\DeclareMathOperator{\id}{id}
\DeclareMathOperator{\pr}{pr}

\title{\bf Moduli of coisotropic sections and the BFV-complex}

\author[F. Sch\"atz]{Florian Sch\"atz}
\address{Institut f\"ur Mathematik, Universit\"at Z\"urich--Irchel, Winterthurerstrasse 190, CH-8057 Z\"urich, Switzerland}
\email{florian.schaetz@math.uzh.ch}

\thanks{The author acknowledges partial support by the research grant of the University of Z\"urich,
by the SNF-grant 200020-121640/1,
by the European Union through the FP6 Marie Curie RTN ENIGMA (contract number MRTN-CT-2004-5652), 
and by the European Science Foundation through the MISGAM program.}

\begin{document}

\maketitle 

\begin{abstract}
We consider the local deformation problem of coisotropic submanifolds inside Poisson manifolds.
To this end the groupoid of coisotropic sections (with respect to some tubular neighbourhood) is introduced.
Although the geometric content of this groupoid is evident, it is usually a very intricate object.

We provide a description of the groupoid of coisotropic sections in terms of a differential graded Poisson
algebra, called the BFV-complex. This description is achived by
constructing a groupoid from the BFV-complex and a surjective morphism from this groupoid to the groupoid of coisotropic sections.
The kernel of this morphism can be easily chracterized.

As a corollary we obtain an isomorphism between the moduli space of coisotropic sections
and the moduli space of geometric Maurer--Cartan elements of the BFV-complex.
\end{abstract}

\tableofcontents

\begin{Ack}
First of all I thank Alberto Cattaneo for his enduring support and inspiration. I thank James Stasheff
for valuable discussions and insightful comments on a draft of this paper.
Moreover I thank Peter Michor and Alan Weinstein for
helpful communciation. Last but not least
I am in debt of Thomas Preu for countless discussions on various matters.
\end{Ack}

\section{Introduction}\label{s:intro}
It is well-known that the nearby deformations of a Lagrangian submanifold $L$ inside a symplectic manifold $(M,\omega)$
are controlled by the first de~Rham cohomology $H^{1}(L,\mathbb{R})$ of $L$. Due to
the Darboux--Weinstein theorem (\cite{WeinsteinD}) one can replace $L\hookrightarrow (M,\omega)$ by
$L\hookrightarrow (T^{*}L,\omega_{\can})$ where $\omega_{\can}$ is the standard symplectic structure on $T^{*}L$. 
Graphs of sections of $T^{*}L \to L$
are Lagrangian if and only if they are closed with respect to the de~Rham differential. Moreover
it is possible to prove that two such sections can be connected by a one-parameter family of Hamiltonian diffeomorphisms
if and only if their cohomology classes coincide.

One can try to generalize this situation in two ways. On the one hand one
can try to incorporate effects of global nature. This
is usually done by ``counting'' suitable pseudoholomorphic objects. This idea goes back to Floer (\cite{Floer})
and was developed to a full-fledged theory in recent years -- see \cite{FOOO} for instance.

On the other hand one can try to understand the local deformation problem of more general objects than Lagrangian submanifolds.
A natural class of submanifolds containing Lagrangian ones is given by coisotropic submanifolds. The notion
of coisotropic submanifolds can be easily extended to Poisson geometry and it constitutes a very interesting class
of subobjects there, see \cite{Weinstein}.

The study of the nearby deformations of coisotropic submanifolds inside symplectic manifolds was started by Zambon.
In \cite{Zambon} it was shown that the space of nearby deformations of a given coisotropic submanifold does not always
carry the structure of a (infinite-dimensional) manifold. In \cite{OhPark} this result was explained in terms of
a structure called ``strong homotopy Lie algebroid''. This notion refers to an enrichment of the Lie algebroid complex associated
to a coisotropic submanifold. It was proved that nearby deformations of coisotropic submanifolds are in one-to-one
correspondence to solutions of a certain equation, called the Maurer--Cartan equation,
which is naturally associated to the strong homotopy Lie algebroid.

As remarked in \cite{Schaetz} this correspondence fails for Poisson manifolds. Moreover,
it was shown that another structure which is tightly related to the strong homotopy Lie algebroid
can be used to restore the correspondence between nearby deformations and solutions of the Maurer--Cartan equation (up
to certain equivalences). This other structure is known as the ``BFV-complex'', originally interoduce in
\cite{BatalinFradkin} and \cite{BatalinVilkovisky} respectively.

Here we incorporate the fact that certain deformations of a coisotropic submanifold
should be considered equivalent. Roughly speaking, two deformations are equivalent whenever
they can be connected by a smooth one-parameter family of Hamiltonian diffeomorphisms. We are interested in the set of
nearby deformations of a fixed coisotropic submanifold modulo equivalences.
It turns out that this set can be realized as the set of isomorphism classes of objects in a certain groupoid $\mathcal{C}(E,\Pi)$.

On the other hand we provide a certain set of Maurer--Cartan elements of the BFV-complex -- called geometric Maurer--Cartan
elements -- that can be equipped with the structure of a groupoid which we denote by $\mathcal{D}_{\geo}(E,\Pi)$.
Furthermore there is a surjective morphism of groupoids from $\mathcal{D}_{\geo}(E,\Pi)$ to $\mathcal{C}(E,\Pi)$.
The kernel of this morphism can be easily characterized and consequently we obtain
a description of $\mathcal{C}(E,\Pi)$ in terms of the BFV-complex 
(Theorem \ref{thm:main} Subsection \ref{ss:isomorphism}).
This also yields an isomorphism
between the set of deformations of a fixed coisotropic submanifold up to equivalence and the set
of certain Maurer--Cartan elements up to an adapted version of gauge-equivalence (Theorem \ref{thm:equivalenceclasses} Subsection \ref{ss:isomorphism_moduli}).

Let us brifely summerize the structure of this paper:

In Section \ref{s:moduli_geometric} the appropriate equivalence relation on the set of coisotropic submanifolds
is introduced. Then the set of equivalence classes with respect to this equivalence relation
is realized as the orbit set of a small groupoid.

In Section \ref{s:moduli_algebraic} we review the construction of the BFV-complex.
Two special classes of Maurer--Cartan elements -- the ``normalized'' and the ``geometric'' ones -- are investigated.
The set of geometric Maurer--Cartan elements is equipped with a groupoid structure.
This groupoid comes along with a full normal
subgroupoid and the BFV-groupoid is defined to be the corresponding quotient groupoid.

Section \ref{s:isomorphism} provides the link between the BFV-complex and geometry: Theorem \ref{thm:main}
in Subsection \ref{ss:isomorphism} asserts that the groupoid associated to the nearby deformations of a coisotropic submanifold
is isomorphic to the BFV-groupoid. In particular their orbit spaces are isomorphic
and hence the moduli space of nearby deformations of coisotropic submanifolds is isomorphic to the moduli space
of geometric Maurer--Cartan elements (Theorem \ref{thm:equivalenceclasses}
in Subsection \ref{ss:isomorphism_moduli}).

We hope to apply the theoretical developments presented here in a more concrete fashion in the future. In particular
it would be interesting to see whether it is possible to 
derive stability conditions for a given coisotropic submanifold in terms of the BFV-complex.

\section{The moduli space of coisotropic sections}\label{s:moduli_geometric}
We briefly review basic definitions of Poisson geometry and coisotropic submanifolds thereof in Subsection \ref{ss:coisotropicsections}.
Moreover coisotropic vector bundles and coisotropic sections are introduced. In Subsection \ref{ss:Hamiltonianhomotopies}
a certain equivalence relation $\sim_{H}$ on the set of coisotropic sections is defined. The set of equivalence classes with
respect to $\sim_{H}$ is a good
candidate for the moduli space of coisotropic sections of a coisotropic vector bundle.
In Subsection \ref{ss:groupoidofcoisotropicsections} a groupoid which provides a refinement of $\sim_{H}$ is constructed. Finally
a short summary of results in relation to this groupoid is given. Here we mostly follow \cite{WeinsteinD}, \cite{Zambon}, \cite{OhPark},
\cite{CattaneoFelder} and \cite{Schaetz}.

\subsection{Coisotropic sections}\label{ss:coisotropicsections}

Let $M$ be a smooth finite dimensional manifold equipped with a {\em Poisson bivector field} $\Pi$, i.e.
a section of $\wedge^{2}TM$ that satisfies $[\Pi,\Pi]_{SN}=0$. Here $[\cdot,\cdot]_{SN}$ denotes
the Schouten--Nijenhuis bracket on $\Gamma(\wedge TM)$. Instead of $\Pi$ one can consider the associated {\em Poisson bracket}
$\{\cdot,\cdot\}_{\Pi}$ which is a biderivation of the algebra of smooth functions $\mathcal{C}^{\infty}(M)$ on $M$. 
It is given
by
\begin{align*}
\{f,g\}_{\Pi}:=<\Pi,df\wedge dg>
\end{align*}
where $<\cdot,\cdot>$ is the natural contraction between $\wedge TM$ and $\wedge T^{*}M$.
The condition $[\Pi,\Pi]_{SN}=0$ is equivalent to $(\mathcal{C}^{\infty}(M),\{\cdot,\cdot\}_{\Pi})$ 
being a Lie algebra.

Given a Poisson bivector field $\Pi$, one defines a bundle map via
\begin{align*}
\Pi^{\#}:T^{*}M \to TM, \quad \alpha \mapsto <\Pi,\alpha>.
\end{align*}

Let $S$ be a submanifold of $M$.
The {\em conormal bundle} $N^{*}S$ of $S$ in $M$ is defined via the following short exact sequence
of vector bundles over $S$:

$$
\xymatrix{
0 \ar[r] & N^{*}S \ar[r] & T^{*}_{S}M \ar[r] & T^{*}S \ar[r] & 0}.
$$

\begin{Definition}\label{coisotropic_submanifold}
A submanifold $S$ of a smooth finite dimensional Poisson manifold $(M,\Pi)$
is called {\em coisotropic} if the restriction of $\Pi^{\#}$ to $N^{*}S$
has image in $TS$.
\end{Definition}

An equivalent definition can be given in terms of the {\em vanishing ideal} of $S$ in $M$ defined by
\begin{align*}
\mathcal{I}_S:=\{f\in \mathcal{C}^{\infty}(M): f|_{S}\equiv 0\}.
\end{align*}
The multiplicative ideal $\mathcal{I}_S$ is called a {\em coisotrope} in $\mathcal{C}^{\infty}(M)$ if it is a Lie subalgebra
of $(\mathcal{C}^{\infty}(M),\{\cdot,\cdot\}_{\Pi})$. Given a submanifold $S$,
it is straightforward to check that $\mathcal{I}_{S}$ is a coisotrope
if and only if $S$ is a coisotropic submanifold of $(M,\Pi)$, see \cite{Weinstein} for details.

The {\em normal bundle} $NS$ of $S$ in $M$ is defined via the following short exact sequence
of vector bundles over $S$:

$$
\xymatrix{
0 \ar[r] & TS \ar[r] & T_{S}M \ar[r] & NS \ar[r] & 0.}
$$

It is well-known that an embedding of $NS$ into $M$ exists such that its restriction to $S$ coincides with the identity.
Using such an embedding, $NS$ inherits a Poisson bivector field from $M$. Since we are interested in the local
properties near $S$ only, we will fix such an embedding once and for all. Consequently our main object of study is

\begin{Definition}\label{def:coisotropicvectorbundle}
A {\em coisotropic vector bundle} is a pair $(E,\Pi)$ such that
\begin{itemize}
\item[(a)] $E\xrightarrow{p} S$ is a finite rank vector bundle over a smooth finite dimensional manifold $S$,
\item[(b)] $\Pi$ is a Poisson bivector field on $E$,
\item[(c)] $S$ embedded into $E$ as the zero section is a coisotropic submanifold of $(E,\Pi)$.
\end{itemize}
\end{Definition}

Next we define 

\begin{Definition}\label{def:coisotropicsections}
Let $(E,\Pi)$ be a coisotropic vector bundle. A section $\mu$ of $E$ is {\em coisotropic} if
its graph is a coisotropic submanifold of $(E,\Pi)$.
We denote the set of all coisotropic sections of $(E,\Pi)$ by $\mathcal{C}(E,\Pi)$.
\end{Definition}

\subsection{Hamiltonian homotopies}\label{ss:Hamiltonianhomotopies}

Every coisotropic vector bundle $(E,\Pi)$ (in fact every Poisson manifold) comes along with a natural group of inner automorphisms,
the group of {\em Hamiltonian diffeomorphisms}. These are diffeomorphisms generated by a 
time-dependent family of {\em Hamiltonian vector fields},
i.e. vector fields of the form $X_{f}:=\Pi^{\#}(df)$ for some smooth function $f$ on $E$. To be more precise a diffeomorphism
$\phi$ of $E$ is called {\em Hamiltonian} if
\begin{itemize}
\item[(a)] there is a smooth map $\hat{\phi}: E\times [0,1] \to E$ such that its restriction $\phi_t$ to
$E\times \{t\}\hookrightarrow E\times [0,1]$ is a diffeomorphism for arbitrary $t\in [0,1]$ and $\phi_0=\id$, $\phi_1=\phi$,
\item[(b)] there is a smooth function $F: E\times [0,1] \to \mathbb{R}$ such that for all $x\in E$ and all $s\in [0,1]$ the equation
\begin{align*}
\frac{d}{dt}|_{t=s}\phi_{t}(x)=X_{F_s}|_{\phi_s(x)}
\end{align*}
holds.
\end{itemize}
A {\em smooth one-parameter family of Hamiltonian diffeomorphisms} is a smooth map $\hat{\phi}: E\times [0,1] \to E$ satifying properties (a) and
(b) from above except that $\phi_1$ is not fixed. We denote the restriction of $\hat{\phi}$ to $E\times \{t\}\cong E$ by $\phi_t$
as above.

The set of Hamiltonian diffeomorphisms $\Ham(E,\Pi)$ forms a subgroup of the set of all {\em Poisson diffeomorphisms}, i.e.
all diffeomorphisms $\psi$ of $E$ such that
\begin{align*}
\{\psi^{*}(\cdot),\psi^{*}(\cdot)\}_{\Pi}=\psi^{*}(\{\cdot,\cdot\}_{\Pi})
\end{align*}
is satisfied.
This implies in particular that Hamiltonian diffeomorphisms map coisotropic submanifolds to coisotropic submanifolds.

Denote the set of smooth one-parameter families of Hamiltonian diffeomorphisms of $(E,\Pi)$ by $\underline{\Ham}(E,\Pi)$.
This set comes along with a natural group structure given by
\begin{align*}
(\hat{\phi},\hat{\psi})\mapsto (\phi_t\cdot \psi_t)_{t\in [0,1]}.
\end{align*}
Furthermore any element of $\underline{\Ham}(E,\Pi)$ maps coisotropic submanifolds to one-parameter families of
coisotropic submanifolds.
To describe the action of $\underline{\Ham}(E,\Pi)$ on the the set of coisotropic sections $\mathcal{C}(E,\Pi)$
we introduce the concept of {\em Hamiltonian homotopies}:

\begin{Definition}\label{def:Hamiltonianhomotopies}
Let $(E,\Pi)$ be a coisotropic vector bundle and $\mu_{0}$ and $\mu_{1}$ two coisotropic sections of $(E,\Pi)$.
A {\em Hamiltonian homotopy} from $\mu_0$ to $\mu_1$ is a pair $(\hat{\mu},\hat{\phi})$ where
\begin{itemize}
\item[(a)] $\hat{\mu}$ is a section of the pull back bundle of $E$ along $S\times [0,1]\to S$ and
\item[(b)] $\hat{\phi}$ is a smooth one-parameter family of Hamiltonian diffeomorphisms of $(E,\Pi)$
\end{itemize}
such that
\begin{itemize}
\item[(a')] the restriction of $\hat{\mu}$ to $S\times \{0\}$ is $\mu_0$ and the restriction to $S\times \{1\}$ is $\mu_1$ and
\item[(b')] for all $t\in [0,1]$ the image of the graph of $\mu_0$ under $\phi_t$ is equal to the 
graph of the restriction of $\hat{\mu}$ to $S\times \{t\}$.
\end{itemize}
\end{Definition}

\begin{Lemma}\label{lemma:Hamiltonianhomotopy}
The relation on $\mathcal{C}(E,\Pi)$ given by
\begin{align*}
\left(\mu \sim_H \nu\right) :\Leftrightarrow \quad \text{there is a Hamiltonian homotopy from } \mu \text{ to } \nu
\end{align*}
is an equivalence relation. 
\end{Lemma}

\begin{proof}
{\em Reflexivity:} Given $\mu$ an arbitrary coisotropic section of $(E,\Pi)$, the pair 
\begin{align*}
\left((\mu)_{t\in [0,1]},(\id_{E})_{t\in [0,1]}\right)
\end{align*}
is a Hamiltonian homotopy from $\mu$ to $\mu$. We denote this Hamiltonian homotopy by $\id_{\mu}$.

{\em Symmetry:} Given a Hamiltonian homotopy $(\hat{\mu},\hat{\phi})$ from $\mu_0$ to $\mu_1$, the pair
\begin{align*}
\left(\mu_{(1-t)}, (\phi_{(1-t)})\circ \phi_1^{-1}\right)_{t\in [0,1]}
\end{align*}
is a Hamiltonian homotopy from $\mu_1$ to $\mu_0$. We denote this Hamiltonian homotopy by $(\hat{\mu},\hat{\phi})^{-1}$.

{\em Transitivity:} Let $(\hat{\alpha},\hat{\phi})$ be a Hamiltonian homotopy from $\mu$ to $\nu$ and 
$(\hat{\beta},\hat{\psi})$ a Hamiltonian homotopy from $\nu$ to $\omega$.

We choose a smooth function $\rho: [0,1] \to [0,1]$ that is strictly increasing and a diffeomorphism on its image on
$[0,1/3[$ and on $]2/3,1]$ respectively. On $[1/3,2/3]$ it $\rho$ is required to be equal to $1/2$. We call any such function a {\em gluing function}.
The {\em composition}
\begin{align*}
(\hat{\beta},\hat{\psi})\Box_{\rho}(\hat{\alpha},\hat{\phi}):=(\hat{\beta}\Box_{\rho}\hat{\alpha},\hat{\psi}\Box_{\rho}\hat{\phi})
\end{align*}
of $(\hat{\alpha},\hat{\phi})$ and $(\hat{\beta},\hat{\psi})$ {\em with respect to} $\rho$
is given by
\begin{align*}
(\hat{\psi}\Box_{\rho}\hat{\phi})(t):=
\begin{cases}
\phi_{2\rho(t)} &0\le t \le 1/3 \\
\phi_1          &1/3 \le t \le 2/3 \\
\psi_{2\rho(t)-1}\circ \phi_1 & 2/3 \le t \le 1
\end{cases}
\quad \quad \text{ and }
\end{align*}
\begin{align*}
(\hat{\beta}\Box_{\rho}\hat{\alpha})(t):=
\begin{cases}
\alpha_{2\rho(t)} & 0\le t \le 1/3 \\
\alpha_1=\nu=\beta_0          & 1/3 \le t \le 2/3 \\
\beta_{2\rho(t)-1} & 2/3 \le t \le 1
\end{cases}
\quad \quad \text{ respectively.}
\end{align*}
It is straightforward to check that $(\hat{\beta},\hat{\psi})\Box_{\rho}(\hat{\alpha},\hat{\phi})$
is a Hamiltonian homotopy from $\mu$ to $\omega$.
\end{proof}

\begin{Definition}\label{def:moduliofcoisotropicsections}
Let $(E,\Pi)$ be a coisotropic vector bundle.
We denote the set of equivalence classes of coisotropic sections under $\sim_{H}$
by $\mathcal{M}(E,\Pi)$ and call it the set of {\em coisotropic sections} of $(E,\Pi)$ {\em modulo Hamiltonian homotopies}
or the {\em moduli space of coisotropic sections} of $(E,\Pi)$.
\end{Definition}

\subsection{The groupoid of coisotropic sections}\label{ss:groupoidofcoisotropicsections}

We want to construct a groupoid whose set of orbits is equal to the moduli space of coisotropic sections $\mathcal{M}(E,\Pi)$.
The main problem is that there is no ``natural'' composition on the set of Hamiltonian homotopies (with matching data at the
end of the first one and at the beginning of the second one respectively). The operation $\Box_{\rho}$
depends on a choice of a gluing function $\rho$ and is not associative. To overcome these problems
we introduce an equivalence relation on the set of Hamiltonian homotopies:

\begin{Definition}\label{def:isotopiesofHamiltonianhomotopies}
Let $(E,\Pi)$ be a coisotropic vector bundle.
An {\em isotopy of Hamiltonian homotopies} is a pair $(\hat{\mu},\hat{\Phi})$ where
\begin{itemize}
\item[(a)] $\hat{\mu}$ is a section of the pull back bundle of $E$ along $S\times [0,1]^{2}\to S$ and
\item[(b)] $\hat{\Phi}$ is a smooth mapping $E\times [0,1]^{2} \to E$ whose restriction to $E\times \{t\}\times \{s\}$
is a diffeomorphism for arbitrary $s,t \in [0,1]$
\end{itemize}
such that
\begin{itemize}
\item[(a')] the restriction of $\hat{\Phi}$ to $E\times \{0\} \times [0,1]$ is equal to $(\id_{E})_{s\in [0,1]}$,
\item[(b')] the restriction of $\hat{\mu}$ to $E\times \{1\} \times [0,1]$ is constant in $s\in I$,
\item[(c')] there is a smooth function $F: E\times [0,1]^{2}\to \mathbb{R}$ such that
the restriction of $\hat{\Phi}$ to $E \times I \times \{s\}$ is the smooth one-parameter family of Hamiltonian diffeomorphisms
generated by the restriction of $F$ to $E\times I \times \{s\}$ and
\item[(d')] the image of the graph of $\mu_{0,s}$ under $\Phi_{t,s}$ is equal to the 
graph of the $\mu_{t,s}$ for all $s,t\in [0,1]$.
\end{itemize}.

We say that an isotopy of Hamiltonian homotopies $(\hat{\mu},\hat{\Phi})$
starts at the Hamiltonian homotopy $(\hat{\mu}|_{S\times [0,1] \times \{0\}},\hat{\Phi}|_{E\times [0,1] \times \{0\}})$
and ends at the Hamiltonian homotopy $(\hat{\mu}|_{S\times [0,1] \times \{1\}},\hat{\Phi}|_{E\times [0,1] \times \{1\}})$.
\end{Definition}

\begin{Lemma}\label{lemma:isotopiesofHamiltonianhomotopies}
\begin{itemize}
\item[(a)]
The relation on the set of Hamiltonian homotopies given by
\begin{align*}
(\hat{\mu},\hat{\phi}) \simeq_{H} (\hat{\nu},\hat{\psi}) :\Leftrightarrow
\end{align*}
there is an isotopy of Hamiltonian homotopies from $(\hat{\mu},\hat{\phi})$ to $(\hat{\nu},\hat{\psi})$;
defines an equivalence relation on the set of Hamiltonian homotopies.
\item[(b)] Let $\rho$ and $\rho'$ be two gluing functions. Then the compositions
of Hamiltonian homotopies with respect to $\rho$ and $\rho'$ coincide
up to $\simeq_{H}$.
\item[(c)] The Hamiltonian homotopies
\begin{align*}
\id_{\mu_0}\Box_{\rho}(\hat{\mu},\hat{\phi}) \text{ and } (\hat{\mu},\hat{\phi})\Box_{\rho}\id_{\mu_1}
\end{align*}
are equivalent to $(\hat{\mu},\hat{\phi})$.
\item[(d)]
The Hamiltonian homotopies 
\begin{align*}
(\hat{\mu},\hat{\phi})^{-1}\Box_{\rho}(\hat{\mu},\hat{\phi}) \text{ and } (\hat{\mu},\hat{\phi})\Box_{\rho}(\hat{\mu},\hat{\phi})^{-1}
\end{align*}
are equivalent to $\id_{\mu_{0}}$.
\item[(e)]
The operation $\Box_{\rho}$ defined in the proof of Lemma \ref{lemma:Hamiltonianhomotopy}
descends to the set of Hamiltonian homotopies modulo isotopies of Hamiltonian homotopies and is associative there.
\end{itemize}
\end{Lemma}

\begin{proof}

{\em (a):}
The proof can be copied from the proof of Lemma \ref{lemma:Hamiltonianhomotopy}. In particular one
makes use of the fact that isotopies of Hamiltonian homotopies can be composed along $I \times \{1\}$
and $I\times \{0\}$ respectively if the data attached to the boundaries match.

{\em (b):}
Choose a smooth function $\tau: [0,1]\to [0,1]$ that is $0$ on $[0,1/3]$, $1$ on $[2/3,1]$ and a diffeomorphism
on $]1/3,2/3[$.
Consider the smooth one-parameter family of gluing functions $\hat{\rho}(s):=(1-\tau(s))\rho + \tau(s) \rho'$. Then
\begin{align*}
(\hat{\nu},\hat{\psi})\Box_{\hat{\rho}(s)}(\hat{\mu},\hat{\phi})
\end{align*}
is an isotopy of Hamiltonian homotopies from
$(\hat{\nu},\hat{\psi})\Box_{\rho}(\hat{\mu},\hat{\phi})$
to $(\hat{\nu},\hat{\psi})\Box_{\rho'}(\hat{\mu},\hat{\phi})$.

{\em (e):}
Choose a smooth function $\tau: [0,1]\to [0,1]$ that is $0$ on $[0,1/3]$, $1$ on $[2/3,1]$ and a diffeomorphism
on $]1/3,2/3[$. Setting $((\hat{\mu},\hat{\phi})\Box_{\rho}\id_{\mu_1})\circ g_{s}(t)$ with
\begin{align*}
g_{s}(t):=(1-\frac{2}{3}(1-\tau(s)))t
\end{align*}
yields an isotopy of Hamiltonian homotopies from $(\hat{\mu},\hat{\phi})\Box_{\rho}\id_{\mu_1}$ to $(\mu_{\alpha(t)},\hat{\phi}_{\alpha(t)})$
where $\alpha$ is a diffeomorphism of $[0,1]$ relative to the boundary. Now
\begin{align*}
(\mu_{((1-\tau(s))\alpha(t)+\tau(s)t)},\phi_{((1-\tau(s))\alpha(t)+\tau(s)t)})
\end{align*}
is an isotopy of Hamiltonian homotopies from $(\mu_{\alpha(t)},\phi_{\alpha(t)})$ to $(\hat{\mu},\hat{\phi})$.
Since $\simeq_{H}$ is an equivalence relation we obtain $(\hat{\mu},\hat{\phi})\Box_{\rho}\id_{\mu_{1}}\simeq_{H}(\hat{\mu},\hat{\phi})$.
Similarly one finds an isotopy of Hamiltonian homotopies from $\id_{\mu_{0}}\Box_{\rho}(\hat{\mu},\hat{\phi})$ to
$(\hat{\mu},\hat{\phi})$.

{\em (d):}
The Hamiltonian homotopy $(\hat{\mu},\hat{\phi})^{-1}\Box_{\rho}(\hat{\mu},\hat{\phi})$ is given by
\begin{align*}
\begin{cases}
\mu_{2\rho(t)} & 0\le t \le 1/3 \\
\mu_1          & 1/3 \le t \le 2/3 \\
\mu_{2(1-\rho(t))} & 2/3 \le t \le 1
\end{cases}
\quad ,
\begin{cases}
\phi_{2\rho(t)} & 0\le t \le 1/3 \\
\phi_1          & 1/3 \le t \le 2/3 \\
\phi_{2(1-\rho(t))}\circ \phi_1 & 2/3 \le t \le 1
\end{cases}
.
\end{align*}
Choose a smooth function $\sigma(s)$ from $[0,1]$ to $[0,1]$ which is $1$ for $s\le 1/3$ and vanishes for $s\ge 2/3$.
The following is an isotopy of Hamiltonian homotopies 
from $(\hat{\mu},\hat{\phi})^{-1}\Box_{\rho}(\hat{\mu},\hat{\phi})$ to $\id_{\mu_0}$:
\begin{align*}
\begin{cases}
\mu_{2\rho(t)\sigma(s)} & 0\le t \le 1/3 \\
\mu_{\sigma(s)}          & 1/3 \le t \le 2/3 \\
\mu_{2(1-\rho(t))\sigma(s)} & 2/3 \le t \le 1
\end{cases}
\quad ,
\begin{cases}
\phi_{2\rho(t)\sigma(s)} & 0\le t \le 1/3 \\
\phi_{\sigma(s)}          & 1/3 \le t \le 2/3 \\
\phi_{2(1-\rho(t))\sigma(s)}\circ \phi_1 & 2/3 \le t \le 1
\end{cases}
.
\end{align*}
For $(\hat{\mu},\hat{\phi})\Box_{\rho}(\hat{\mu},\hat{\phi})^{-1}$ an isotopy of Hamiltonian homotopies to $\id_{\mu_{0}}$
can be found in the same
fashion.

{\em (e):}
That the composition of two Hamiltonian homotopies with respect to some gluing function $\rho$ descends to the set of
equivalence classes
of $\simeq_{H}$ is implied by the fact that
isotopies of Hamiltonian homotopies might be glued along their boundary strata $\{1\} \times [0,1]$ and $\{0\}\times [0,1]$ respectively
if the data attached to the boundaries match.

The associativity of $\Box{\rho}$ on the set of equivalence classes $\simeq_{H}$ is proved as follows:
Let $(\hat{\alpha},\hat{\phi})$ be a Hamiltonian homotopy from $\mu$ to $\nu$, 
$(\hat{\beta},\hat{\psi})$ a Hamiltonian homotopy from $\nu$ to $\omega$ and
$(\hat{\gamma},\hat{\varphi})$ a Hamiltonian homotopy from $\omega$ to $\Omega$. We have to find an isotopy of Hamiltonian homotopies
from
\begin{align*}
A:=(\hat{\gamma},\hat{\varphi})\Box_{\rho}\left((\hat{\beta},\hat{\psi}) \Box_{\rho} (\hat{\alpha},\hat{\phi})\right)
\end{align*}
to
\begin{align*}
B:=\left((\hat{\gamma},\hat{\varphi}) \Box_{\rho} (\hat{\beta},\hat{\psi})\right) \Box_{\rho}(\hat{\alpha},\hat{\phi}).
\end{align*}
First we choose a smooth one-parameter family of diffeomorphisms $\kappa_{s}$ of $[0,1]$
relative to the boundary $\{0\}\cup \{1\}$ such that
$\kappa_{0}=\id$ and $\kappa_{1}$ maps the interval $[1/5,2/5]$ to $[1/9,2/9]$ and $[3/5,4/5]$ to $[1/3,2/3]$. We extend
the Hamiltonian homotopy $A$ to an isotopy of Hamiltonian homotopies by composing with $\kappa_{s}$. Analogously
one reparametrizes $B$ by an isotopy of Hamiltonian homotopies such that $[1/5,2/5]$ and 
$[3/5,4/5]$ get mapped to $[1/3,2/3]$ and $[7/9,8/9]$ respectively.
The two resulting Hamiltonian homotopies can be joined in an ``affine'' manner using the function $\tau$ from part (b).
Since $\simeq_{H}$ is an equivalence relation,
these three isotopies of Hamiltonian homotopies (reparametrization of $A$, reparametrization of $B$ and affine connection between
the reparametrized Hamiltonian homotopies) can be glued together to yield an isotopy of Hamiltonian homotopies
from $A$ to $B$. 
\end{proof}

\begin{Definition}\label{def:groupoidofcoisotropicsections}
Let $(E,\Pi)$ be a coisotropic vector bundle. The {\em groupoid of coisotropic sections} of $(E,\Pi)$, which we denote
by $\hat{\mathcal{C}}(E,\Pi)$, is the small groupoid where
\begin{itemize}
\item[(a)] the set of objects is the set of coisotropic sections $\mathcal{C}(E,\Pi)$ of $(E,\Pi)$,
\item[(b)] the set of morphisms $\Hom(\mu,\nu)$ between two coisotropic sections $\mu$ and $\nu$ is the set
of all Hamiltonian homotopies from $\mu$ to $\nu$ modulo isotopies of Hamiltonian homotopies and
\item[(c)] the composition is induced from composition of Hamiltonian homotopies with respect to some gluing function.
\end{itemize}
\end{Definition}

Recall that {\em small groupoid} is a groupoid whose objects and morphisms form honest sets and not just classes.

It seems very likely that the groupoid $\hat{\mathcal{C}}(E,\Pi)$ can be understood as a truncation of a
weak $\infty$-groupoid $\hat{\mathcal{C}}^{\infty}(E,\Pi)$ at its two-morphisms which should be 
given by isotopies of Hamiltonian homotopies.
In fact the two ways of gluing
isotopies of Hamiltonian homotopies that were used in the proof of Lemma \ref{lemma:isotopiesofHamiltonianhomotopies}
should correspond to vertical and horizontal composition of two-morphisms.

The set of orbits of $\hat{\mathcal{C}}(E,\Pi)$ is the moduli space of coisotropic sections $\mathcal{M}(E,\Pi)$
of $(E,\Pi)$ modulo Hamiltonian homotopies.
We give a short overview of known results related to this object.

Under the assumption that $S$ is a {\em Lagrangian submanifold} of a {\em symplectic manifold}, any
embedding of $NS\cong T^{*}S=E$ into $M$ yields a Poisson structure on $E$ which is symplectomorph to the natural symplectic structure
on some open neighbourhood $U$ of $S$ in $E$, see \cite{WeinsteinD}.
This allows us to reduce the nearby deformation problem of $L$ in $(M,\omega)$ to the case $L\hookrightarrow (T^{*}L,\omega_{\can})$.
Hence $L\hookrightarrow (T^{*}L,\omega_{\can})$ is a ``universal model'' of $L$ as a Lagrangian submanifold of a symplectic manifold,
as far as local properties are concerned.
For $L\hookrightarrow (T^{*}L,\omega_{\can})$ the following facts are well-known:
\begin{itemize}
\item[(a)] the set of coisotropic sections of $(T^{*}L,\omega_{\can})$ is isomorphic to the set of closed
one forms on $L$,
\item[(b)] two coisotropic sections of $(T^{*}L,\omega_{\can})$ are related by a Hamiltonian homotopy if and only if
their classes in de~Rham cohomology coincide,
\item[(c)] the space of coisotropic sections modulo Hamiltonian homotopies is isomorphic to
$H^{1}(L,\mathbb{R})$.
\end{itemize}
The implication ($\Rightarrow$) in (b) needs some additional argument using the symplectic action of a path inside an exact symplectic
manifold, see \cite{McDuffSalamon}
for instance.

The case of a {\em coisotropic submanifold} $S$ of a {\em symplectic manifold} was studied by Zambon (\cite{Zambon})
and Oh and Park (\cite{OhPark}). Zambon investigated the set of coisotropic sections 
and proved that it does not carry a reasonable
structure of a (Frech\'et-)manifold in general. This observation was explained by Oh and Park in terms of their {\em strong
homotopy Lie algebroid}. The idea is to consider the
{\em Lie algebroid complex} $(\Gamma(\wedge E),\partial)$ of $S$ in $(E,\Pi)$
that is an appropriate replacement of the complex $(\Omega(S),d_{DR})$. They constructed
higher order operations on $\Gamma(\wedge E)$
and identified coisotropic sections of $(E,\Pi)$ contained in some open neighbourhood $U$ of $S$ in $E$
with special elements of $\Gamma(E)$ contained in $U$
that satisfy a generalization
of the closedness condition $\partial(\alpha)=0$. To be more precise,
Oh and Park equipped $\Gamma(\wedge E)$ with the structure of an $L_{\infty}$-algebra compatible with $\partial$
and proved that Maurer--Cartan elements of this structure which are contained in $U$ are exactly the coisotropic sections of $(E,\Pi)$
which are contained in $U$. This construction implies that the formal neighbourhood of $S$ in the
space of coisotropic sections is not necessarily a vector space which explains Zambon's observation.

Cattaneo and Felder (\cite{CattaneoFelder}) extended the construction of the $L_{\infty}$-algebra structure on $\Gamma(\wedge E)$ to the Poisson case.
However the connection between coisotropic sections on the one hand and Maurer--Cartan elements on the other hand
as found by Oh and Park in the symplectic
setting does not hold in the Poisson setting. See \cite{Schaetz} for an example of a coisotropic submanifold of a Poisson manifold
for which the strong homotopy Lie algebroid fails
to detect obstruction to deformations in any open neighbourhood of the coisotropic submanifold, i.e. there are far more solutions
of the Maurer--Cartan equation than coisotropic sections. In \cite{Schaetz} an appropriate replacement of the strong
homotopy Lie algebroid was presented. It is a differential graded Poisson algebra known as the BFV-complex. Furthermore
it was proved that the set of coisotropic sections $\mathcal{C}(E,\Pi)$ of $(E,\Pi)$ is isomorphic to the set of certain equivalence
classes
of normalized Maurer--Cartan elements of the BFV-complex. The situation will be reviewed in more detail in the next Section.

We remark that in the case of a coisotropic submanifold inside a symplectic manifold
a complete description of $\mathcal{M}(E,\Pi)$ in terms of the strong homotopy Lie
algebroid is missing, although we expect that the arguments used in the Lagrangian case could be adapted. 
Moreover, even in case of a Lagrangian submanifold inside a symplectic manifold one obtains an isomorphism on the level
of equivalence classes of coisotropic section under $\sim_H$ but not a description of the groupoid $\hat{\mathcal{C}}(E,\Pi)$
itself.

\section{The BFV-groupoid}\label{s:moduli_algebraic}
Given a coisotropic vector bundle $(E,\Pi)$ supplemented by a choice of auxiliary data,
one can construct a certain differential graded Poisson algebra, called a BFV-complex for $(E,\Pi)$.
We review this construction in Subsection \ref{ss:BFV-complex}. Every differential Lie algebra
comes along with a group of inner automorphisms, which we spell out for the special case of the
BFV-complex in Subsection \ref{ss:gaugeaction}. Furthermore there is a set of distinguished elements of the BFV-complex,
consisting of those elements which satisfy the Maurer--Cartan equation. The group of inner automorphisms acts
on this set. We need to restrict our attention to certain classes
of Maurer--Cartan elements: the ``normalized'' ones (Subsection \ref{ss:normalizedMC-elements}) and
the ``geometric'' ones (Subsection \ref{ss:geometricMC-elements}). Both classes are connected to the 
geometry of the underlying coisotropic vector bundle (Theorem \ref{thm:normalizedMCelements} and \ref{thm:geometricMCelements}).
In Subsection \ref{ss:gaugehomotopies} an equivalence relation $\sim_{G}$ on the set of geometric
Maurer--Cartan elements of the BFV-complex is defined. A groupoid $\hat{\mathcal{D}}(E,\Pi)$ is constructed
whose set of orbits is equal to the set of equivalence classes with respect to $\sim_{G}$. This groupoid is
the quotient of a groupoid $\hat{\mathcal{D}}_{\geo}(E,\Pi)$ (Subsection \ref{ss:geometricBFV-groupoid})
by a full normal subgroupoid to be introduced in Subsection \ref{ss:BFV-groupoid}.

\subsection{The BFV-complex}\label{ss:BFV-complex}

The BFV-complex was originally introduced by Batalin, Fradkin and Vilkovisky (\cite{BatalinFradkin},\cite{BatalinVilkovisky})
in order to understand physical systems with complicated symmetries. Later on this construction was given an interpretation
in terms of homological algebra by Stasheff (\cite{Stasheff}). In the smooth setting a convenient globalization was
found by Bordemann and Herbig (\cite{Bordemann}, \cite{Herbig}). In \cite{Schaetz} Bordemann and Herbig's approach is put into a
more conceptual framework,
in particular a conceptual construction of the BFV-bracket is given. One of the advantages of this conceptual approach
is that it allows us to understand the dependence of the BFV-complex on certain choices involved in its construction (\cite{Schaetz2}).

Consider a Poisson manifold $(E,\Pi)$ where $E\to S$ is a vector bundle. Let $\mathcal{E}\to E$
be the pull back of $E\to S$ along $E\to S$, i.e. the vector bundle fitting into the following Cartesian square
$$
\xymatrix{
\mathcal{E} \ar[r] \ar[d] & E \ar[d] &\\
E \ar[r] & S.}
$$
One defines $BFV(E):=\Gamma(\wedge \mathcal{E} \otimes \wedge \mathcal{E}^{*})$ which is a unital bigraded algebra with bigrading
given by
\begin{align*}
BFV^{(p,q)}(E):=\Gamma(\wedge^{p} \mathcal{E} \otimes \wedge^{q} \mathcal{E}^{*}).
\end{align*}
In physical terminology $p$ ($q$) is referred to as the {\em ghost degree} ({\em ghost-momentum degree}). Moreover the decomposition
of $BFV(E)$ by
\begin{align*}
BFV^{k}(E):=\oplus_{p-q=k}BFV^{(p,q)}(E)
\end{align*}
equips $BFV(E)$ with the structure of a graded algebra. We refer to $k$ as the {\em total degree}.
There is yet another decomposition of $BFV(E)$ that will be useful later: for arbitrary $r\in \mathbb{N}$ set 
$BFV_{\ge r}(E):=\Gamma(\wedge \mathcal{E}\otimes \wedge^{\ge r} \mathcal{E}^{*})$ which is an ideal.
The integer $r$ is called the {\em resolution degree}.

Every choice of connection on $E\to S$ allows us one to construct a graded Poisson bracket on the
graded unital algebra $BFV(E)$, known as the {\em BFV-bracket} (the construction can be found in \cite{Herbig} or in \cite{Schaetz} for instance).
Moreover, the graded Poisson structures on
$BFV(E)$ coming from different choices of connections are all isomorphic (\cite{Schaetz2}).
Hence we choose a connection on $E\to S$ once and for all and denote the corresponding
graded Poisson bracket by $[\cdot,\cdot]_{BFV}$. Independently of the choice of connection we made, $[\cdot,\cdot]_{BFV}$
has the following important properties:

\begin{Lemma}\label{lemma:BFV-bracket}
Let $E\to S$ be a vector bundle equipped with a Poisson bivector field $\Pi$. Choose a connection on $E\to S$ and denote the corresponding
BFV-bracket on $BFV(E)$ by $[\cdot,\cdot]_{BFV}$. Denote the projection from $BFV(E)$ to $BFV^{(0,0)}(E)=\mathcal{C}^{\infty}(E)$
by $P$.
Then $[\cdot,\cdot]_{BFV}$ satisfies the following two properties:
\begin{itemize}
\item[(a)] The restriction of $P\circ [\cdot,\cdot]_{BFV}$ to $\mathcal{C}^{\infty}(E)\times \mathcal{C}^{\infty}(E)$
coincides with $\{\cdot,\cdot\}_{\Pi}$.
\item[(b)] The restriction of $P\circ [\cdot,\cdot]_{BFV}$ to $\Gamma(\mathcal{E})\times \Gamma(\mathcal{E}^{*})$
coincides with the pairing between $\Gamma(\mathcal{E})$ and $\Gamma(\mathcal{E}^{*})$ induced from the natural fiber
pairing between $\mathcal{E}$ and $\mathcal{E}^{*}$.
\end{itemize}
\end{Lemma}

The next step is to find a special degree $+1$ element $\Omega$ of $BFV(E)$ satisfying $[\Omega,\Omega]_{BFV}=0$.
It turns out that there is a normalization condition that makes the choice of such $\Omega$ essentially unique and provides
a tight connection to the geometry of the Poisson manifold $(E,\Pi)$: since
\begin{align*}
BFV^{1}(E)=\oplus_{k\ge 1}\Gamma(\wedge^{k}\mathcal{E}\otimes \wedge^{k-1}\mathcal{E}^{*})
\end{align*}
every degree $+1$ element of $BFV(E)$ has a component in $BFV^{(1,0)}(E)=\Gamma(\mathcal{E})$. Additionally
to $[\Omega,\Omega]_{BFV}=0$ we require that the component of $\Omega$ in $\Gamma(\mathcal{E})$
coincides with the tautological section of the bundle $\mathcal{E}\to E$. Such an element $\Omega$ is called a
{\em BFV-charge}. We denote the tautological section by
$\Omega_0$ from now on.

The following Proposition is contained in \cite{Schaetz2} and the proof essentially follows \cite{Stasheff}:

\begin{Proposition}\label{prop:charge}
Let $E$ be a vector bundle equipped with a Poisson bivector field $\Pi$ and denote its zero section by $S$.
Fix a connection on $E\to S$ and denote the corresponding graded Poisson bracket on $BFV(E)$ by $[\cdot,\cdot]_{BFV}$.

\begin{enumerate}
\item There is a degree $+1$ element $\Omega$ of $BFV(E)$ whose component in $\Gamma(\mathcal{E})$ is given by the tautological
section $\Omega_0$ and that satisfies
\begin{align*}
[\Omega,\Omega]_{BFV}=0
\end{align*}
if and only if $S$ is a coisotropic submanifold of $(E,\Pi)$, i.e. $(E,\Pi)$ is a coisotropic vector bundle.
\item Let $\Omega$ and $\Omega'$ be two BFV-charges. Then there is an automorphism
of the graded Poisson algebra $(BFV(E),[\cdot,\cdot]_{BFV})$ that maps $\Omega$ to $\Omega'$. 
\end{enumerate}
\end{Proposition}

Consequently one can construct a differential graded Poisson algebra of the form
$(BFV(E),[\Omega,\cdot]_{BFV},[\cdot,\cdot]_{BFV})$ for any given coisotropic vector bundle $(E,\Pi)$.
We call such a differential graded Poisson algebra a {\em BFV-complex} 
for $(E,\Pi)$. It is unique up to isomorphisms (\cite{Schaetz2}). For simplicity we fix
1. a connection $\nabla$ on $E\to S$ and 2. a BFV-charge $\Omega$ once and for all and refer to the corresponding BFV-complex as {\em the}
BFV-complex associated to the coisotropic vector bundle $(E,\Pi)$.

\subsection{The gauge group}\label{ss:gaugeaction}
The graded Poisson algebra $(BFV(E),[\cdot,\cdot]_{BFV})$ comes along with a group of inner automorphisms.
We essentially follow \cite{Schaetz} in our exposition but make some definitions more precise.

The subspace
\begin{align*}
BFV^{0}(E)=\oplus_{m\ge 0}\Gamma(\wedge^{m}\mathcal{E}\otimes \wedge^{m}\mathcal{E}^{*}) \subset BFV(E)
\end{align*}
is a unital graded subalgebra of the unital bigraded algebra $BFV(E)$. Moreover it is a Lie subalgebra
of the graded Lie algebra $(BFV(E),[\cdot,\cdot]_{BFV})$. The adjoint action of $BFV(E)$ restricts to
a Lie algebra action of $BFV^{0}(E)$ on $BFV(E)$. This is
the {\em infinitesimal gauge action of} $(BFV(E),[\cdot,\cdot]_{BFV})$.

The graded Poisson algebra $(BFV^{0}(E),[\cdot,\cdot]_{BFV})$ is filtered by a family of graded Poisson algebras
$(BFV^{0}_{\ge r}(E),[\cdot,\cdot]_{BFV})$. Here $BFV^{0}_{\ge r}(E)$ denotes the intersection of
$BFV^{0}(E)$ with the ideal $BFV_{\ge r}(E)$. Lemma \ref{lemma:BFV-bracket} implies that the multiplicative ideals
$BFV^{0}_{\ge r}(E)$ are Poisson subalgebras of $(BFV^{0}(E),[\cdot,\cdot]_{BFV})$. Hence we obtain a filtration
of the infinitesimal gauge action of $(BFV(E),[\cdot,\cdot]_{BFV})$.

Let $\mathcal{E}_{[0,1]}$ be the pull back of $\mathcal{E} \to E$ along $E\times [0,1]\to E$.
We define
\begin{align*}
\widetilde{BFV}(E):=\Gamma(\wedge \mathcal{E}_{[0,1]}\otimes \wedge \mathcal{E}^{*}_{[0,1]})
\end{align*}
which inherits the algebra structure, the bigrading, the total grading, 
the filtration by resolution degree and the graded Poisson bracket
from the corresponding structures on $(BFV(E),[\cdot,\cdot]_{BFV})$. 
In particular the adjoint action restricts to a Lie algebra action
of $\widetilde{BFV}^{0}(E)$ on $\widetilde{BFV}(E)$ and this action
is filtered by actions of $(\widetilde{BFV}^{0}_{\ge r}(E),[\cdot,\cdot]_{BFV})$.
We denote the Lie subalgebra of inner derivations of $\widetilde{BFV}(E)$
coming from the action of $(\widetilde{BFV}^{0}_{\ge r}(E),[\cdot,\cdot]_{BFV})$ by
\begin{align*}
\underline{\mathfrak{inn}}_{\ge r}(BFV(E))
\end{align*}
and set $\underline{\mathfrak{inn}}(BFV(E)):=\underline{\mathfrak{inn}}_{\ge 0}(BFV(E))$.

The {\em group of automorphisms} $\Aut(BFV(E))$ of $(BFV(E),[\cdot,\cdot]_{BFV})$ is the group
of all automorphisms of the unital algebra $BFV(E)$ that preserve the total degree
and the graded Poisson bracket $[\cdot,\cdot]_{BFV}$. An automorphism $\psi$ is called
{\em inner} if it is generated by an element of $\underline{\mathfrak{inn}}(BFV(E))$. More precisely
we impose that
\begin{itemize}
\item[(a)] there is a morphism of unital graded algebras and Poisson algebras
\begin{align*}
\hat{\psi}: (BFV(E),[\cdot,\cdot]_{BFV}) \to (\widetilde{BFV}(E),[\cdot,\cdot]_{BFV})
\end{align*}
and
\item[(b)] there is $\hat{\gamma} \in \widetilde{BFV}^{0}(E)$
\end{itemize}
such that
\begin{itemize}
\item[(a')] the composition $\psi_t$ of $\hat{\psi}$ with the evaluation at $E\times \{t\}$ is an automorphism
of unital graded Poisson algebras for arbitrary $t\in [0,1]$, $\psi_0=id$, $\psi_1=\psi$,
\item[(b')] for all $s\in [0,1]$ and $\beta \in BFV(E)$
\begin{align}\label{ODE}
\frac{d}{dt}|_{t=s}\left(\psi_t(\beta)\right)=-\left([\gamma_s,\psi_s(\beta)]_{BFV}\right)
\end{align}
holds where $\gamma_s$ denotes the restriction of $\hat{\gamma}$ to $E\times \{s\}\cong E$.
\end{itemize}
We remark that this definition is totally analogous to the definition of Hamiltonian diffeomorphisms given 
in Subsection \ref{ss:Hamiltonianhomotopies} if one replaces the one-parameter family
of diffeomorphisms $(\phi_t)_{t \in [0,1]}$ by the corresponding family of push forwards
\begin{align*} 
\left((\phi_t)_{*}:=(\phi_t^{*})^{-1})\right)_{t\in [0,1]}.
\end{align*}

A {\em smooth one-parameter family of inner automorphisms} of the graded Lie algbera $(BFV(E),[\cdot,\cdot]_{BFV})$ is a morphism
$\hat{\psi}$ such as in (a) satisfying (a') and (b') for some $\hat{\gamma}$ as in (b), except that
$\psi_1$ is not fixed. We denote the set of all smooth one-parameter families of inner automorphisms of
$(BFV(E),[\cdot,\cdot]_{BFV})$ by $\underline{\Inn}(BFV(E))$. 
This set comes along with a natural
group structure and the filtration of
$BFV^{0}(E)$ by the Poisson subalgebras $BFV^{0}_{\ge r}(E)$ induces a filtration of
$\underline{\Inn}(BFV(E))$ by subgroups $\underline{\Inn}_{\ge r}(BFV(E))$.

We denote the group of inner automorphisms by $\Inn(BFV(E))$ and the subgroup
generated by elements of $\underline{\mathfrak{inn}}_{\ge r}(BFV(E))$ by $\Inn_{\ge r}(BFV(E))$.

\begin{Lemma}\label{lemma:integrability0}
Any $\hat{\gamma} \in \underline{\mathfrak{inn}}_{\ge 2}(BFV(E))$ can be integrated to a unique
$\hat{\psi} \in \underline{\Inn}_{\ge 2}(BFV(E))$.
\end{Lemma}

\begin{proof}
We have to show that the equation \eqref{ODE}
has a unique solution for arbitrary $\beta \in BFV(E)$ on $[0,1]$.
Lemma \ref{lemma:BFV-bracket} and $\hat{\gamma} \in \underline{\mathfrak{inn}}_{\ge 2}(BFV(E))$
imply that $[-\gamma_s,\cdot]_{BFV}$ is nilpotent because the ghost-momentum degree of this derivation is strictly
positive and the ghost-momentum degree is bounded from above. Hence
existence and uniqueness of a global solution of \eqref{ODE} for $\hat{\gamma} \in \underline{\mathfrak{inn}}_{\ge 2}(BFV(E))$
is implied by the existence and uniqueness
of a flow generated by a smooth one-parameter family of nilpotent vector fields on a finite dimensional supermanifold.
The associated smooth family of inner automorphisms can be written down
explicitly as
\begin{align*}
\phi_t(\cdot):=\exp{\left(\int_{0}^{t}[-\gamma_s,\cdot]_{BFV}ds\right)}
\end{align*}
where $\exp$ refers to the time-ordered exponential.
\end{proof}

\subsection{Normalized Maurer-Cartan elements}\label{ss:normalizedMC-elements}

Let $(E,\Pi)$ be a coisotropic vector bundle and consider the associated differential graded Poisson algebra
\begin{align*}
(BFV(E),[\Omega,\cdot]_{BFV},[\cdot,\cdot]_{BFV}).
\end{align*}
The set of {\em Maurer--Cartan elements} of this differential graded Poisson algebra is
\begin{align*}
\mathcal{D}_{\alg}(E,\Pi):=\{\beta \in BFV^{1}(E): [\Omega+\beta,\Omega+\beta]_{BFV}=0\}.
\end{align*}
It is acted upon by the group of inner automorphisms $\Inn(BFV(E))$ of $(BFV(E),[\cdot,\cdot]_{BFV})$
via 
\begin{align*}
(\psi,\beta)\mapsto \psi\cdot \beta:=\psi(\Omega+\beta)-\Omega.
\end{align*}

We added the subscript ``alg'' because the set $\mathcal{D}_{\alg}(E,\Pi)$ contains elements
that do not possess a clear geometric meaning. Similar to the construction of the BFV-charge
$\Omega$ (Proposition \ref{prop:charge} Subsection \ref{ss:BFV-complex}) one has to add a {\em normalization condition} to make contact
to the geometry of the coisotropic vector bundle $(E,\Pi)$. Since
\begin{align*}
\beta \in BFV^{1}(E)=\oplus_{k\ge 1}\Gamma(\wedge^{k}\mathcal{E}\otimes \wedge^{k-1}\mathcal{E}^{*})
\end{align*}
there is a unique component $\beta_0$ of $\beta$ in $\Gamma(\mathcal{E})$. Recall that $\mathcal{E}\to E$ was defined
to be the pull back of $E\to S$ along $E\xrightarrow{p} S$. Consequently we obtain a pull back map
\begin{align*}
p^{*}: \Gamma(E)\to \Gamma(\mathcal{E}).
\end{align*}

\begin{Definition}\label{def:normalizedMC}
Let $(BFV(E),[\Omega,\cdot]_{BFV},[\cdot,\cdot]_{BFV})$ be
a BFV-complex associated to a coisotropic vector bundle $(E,\Pi)$.

The set of {\em normalized Maurer--Cartan elements} $\mathcal{D}_{\nor}(E,\Pi)$ of $(E,\Pi)$ is the set of
all elements $\beta \in \mathcal{D}_{\alg}(E,\Pi)$ such that 
the component $\beta_0$ of $\beta$ in $\Gamma(\mathcal{E})$ coincides with the pull back
$p^{*}(\mu)$ of some section $\mu\in \Gamma(E)$.
\end{Definition}

Observe that the action of $\Inn(BFV(E))$ does not restrict to an action on $\mathcal{D}_{\nor}(E,\Pi)$.
However, the action of $\Inn_{\ge 2}(BFV(E))$ does.
The map
\begin{align*}
L_{\nor}: \mathcal{D}_{\nor}(E,\Pi) \to \Gamma(E), \quad \beta \mapsto \beta_0=p^{*}(\mu) \mapsto -\mu.
\end{align*}
has the following important properties:

\begin{Theorem}\label{thm:normalizedMCelements}
The map $L_{\nor}$ has the following properties:
\begin{itemize}
\item[(a)] it maps onto the set of coisotropic sections $\mathcal{C}(E,\Pi) \subset \Gamma(\mathcal{E})$
(see Definition \ref{def:coisotropicsections} in Subsection \ref{ss:coisotropicsections}),
\item[(b)] it is invariant under the action of $\Inn_{\ge 2}(BFV(E))$ on $\mathcal{D}_{\nor}(E,\Pi)$,
\item[(c)] it induces an isomorphism
\begin{align*}
[L_{\nor}]: \mathcal{D}_{\nor}(E,\Pi) / \Inn_{\ge 2}(BFV(E)) \xrightarrow{\cong} \mathcal{C}(E,\Pi).
\end{align*}
\end{itemize}
\end{Theorem}

This Theorem was proved in \cite{Schaetz}.

The main aim of the remainder of this paper is to ``lift'' the isomorphism $[L_{\nor}]$ from the level
of sets to the level of appropriate groupoids. On the left-hand side 
$\mathcal{C}(E,\Pi)$ will be replaced by the groupoid of coisotropic sections $\hat{\mathcal{C}}(E,\Pi)$
(see Definition \ref{def:groupoidofcoisotropicsections} in Subsection \ref{ss:groupoidofcoisotropicsections}).
In the following Subsections the right replacement for $\mathcal{D}_{\nor}(E,\Pi)$ with its action of $\Inn_{\ge 2}(BFV(E))$
will be constructed.

\subsection{Geometric Maurer--Cartan elements}\label{ss:geometricMC-elements}

First we prove an extension of Lemma \ref{lemma:integrability0}:

\begin{Lemma}\label{lemma:integrability1}
Any $\hat{\gamma} \in \underline{\mathfrak{inn}}_{\ge 1}(BFV(E))$ can be integrated to a unique
$\hat{\psi} \in \underline{\Inn}_{\ge 1}(BFV(E))$.
\end{Lemma}

\begin{proof}
The decomposition $BFV^{0}_{\ge 1}(E):=\oplus_{m \ge 1}\Gamma(\wedge^{m}\mathcal{E}\otimes \wedge^{m}\mathcal{E}^{*})$
yields a decomposition of $\hat{\gamma}$ into $\hat{A}+\hat{\delta}$ with 
\begin{align*}
\hat{A} \in \Gamma(\mathcal{E}_{[0,1]}\otimes \mathcal{E}^{*}_{[0,1]})
\end{align*}
and $\hat{\delta} \in \underline{\mathfrak{inn}}_{\ge 2}(BFV(E))$.
Lemma \ref{lemma:BFV-bracket} and $\hat{\gamma} \in \underline{\mathfrak{inn}}_{\ge 1}(BFV(E))$ imply that the derivation
 $[-\gamma_s,\cdot]_{BFV}$ can be written as the sum of $-A_s$ acting on $BFV(E)$ by 
the natural fiber pairing between $\wedge\mathcal{E}$ and $\wedge \mathcal{E}^{*}$ plus a nilpotent derivation.
To be more precise the part of $[-\hat{\gamma},\cdot]_{BFV}$ that might not be nilpotent is given as follows:
$-\hat{A}$ is an element of $\Gamma(\mathcal{E}_{[0,1]}\otimes \mathcal{E}^{*}_{[0,1]})=\Gamma(\End(\mathcal{E})_{[0,1]})$.
Here $\End(\mathcal{E})_{[0,1]}$ is the pull back of $\End(\mathcal{E}) \to E$ along $E\times [0,1]\to E$.
As an element of $\Gamma(\End(\mathcal{E})_{[0,1]})$ the family $-\hat{A}$ acts
on $\Gamma(\mathcal{E})$ and this action naturally extends to an action on $\Gamma(\wedge \mathcal{E} \otimes \wedge \mathcal{E}^{*})$.

The smooth one-parameter family $-\hat{A}$ integrates to a unique one-parameter family of fiberwise linear automorphisms
\begin{align*}
\hat{B} \in \Gamma(GL_{+}(\mathcal{E})_{[0,1]}).
\end{align*}
starting at the identity. Here $GL_{+}(\mathcal{E})\to E$ is the bundle of fiberwise linear automorphisms of the vector bundle
$\mathcal{E} \to E$ which are fiberwise orientation preserving.
The family $\hat{B}$ naturally acts on $\Gamma(\wedge (\mathcal{E} \otimes \mathcal{E}^{*}))$.

A straightforward computation shows that equation \eqref{ODE} is equivalent to
\begin{align*}
\frac{d}{dt}|_{t=s}\varphi_t=\left(B_s^{-1}\circ(A_s(\cdot)-[\gamma_s,\cdot]_{BFV})\circ B_{s}\right)(\varphi_s)
\end{align*}
for $\varphi_s:=B_s^{-1}\circ \psi_s$. The endomorphism $A_s(\cdot)-[\gamma_s,\cdot]_{BFV}$ is nilpotent for all $s\in [0,1]$ and so is
\begin{align*}
B_s^{-1}\circ(A_s(\cdot)-[\gamma_s,\cdot]_{BFV})\circ B_{s}.
\end{align*}
Hence the existence and uniqueness of a flow integrating equation \eqref{ODE} for 
$\hat{\gamma} \in \underline{\mathfrak{inn}}_{\ge 1}(BFV(E))$ is equivalent to the existence
and uniqueness of a flow for a smooth one-parameter family of nilpotent vector fields on a finite dimensional supermanifold.
\end{proof}

\begin{Definition}\label{def:geometricMC}
Let $(BFV(E),[\Omega,\cdot]_{BFV},[\cdot,\cdot]_{BFV})$ be
a BFV-complex associated to a coisotropic vector bundle $(E,\Pi)$.

The set of {\em geometric Maurer-Cartan elements} $\mathcal{D}_{\geo}(E,\Pi)$ of $(E,\Pi)$ is the orbit
of $\mathcal{D}_{\nor}(E,\Pi) \subset \mathcal{D}_{\alg}(E,\Pi)$ under the action of
$\Inn_{\ge 1}(BFV(E))$.
\end{Definition}

\begin{Lemma}\label{lemma:geometricMC}
An element $\beta \in \mathcal{D}_{\alg}(E,\Pi)$ is geometric if and only if
there exists $A \in \Gamma(GL_{+}(\mathcal{E}))$ and $\mu\in \Gamma(E)$ such that
\begin{enumerate}
\item[(a)] $\Omega_{0}+\beta_{0}=A(\Omega_{0} + p^{*}(\mu))$ and
\item[(b)] $-\mu$ is a coisotropic section of $(E,\Pi)$.
\end{enumerate}
Moreover given $\beta \in \mathcal{D}_{\geo}(E,\Pi)$, the associated section $\mu\in \Gamma(E)$ is unique.
We denote it by $\mu_{\beta}$ from now on.
\end{Lemma}

\begin{proof}
Let $\beta$ be in $\mathcal{D}_{\geo}(E,\Pi)$. By definition there is $\alpha \in \Inn_{\ge 1}(BFV(E))$
and $\beta' \in \mathcal{D}_{\nor}(E,\Pi)$ such that
\begin{align*}
\Omega + \beta = \alpha(\Omega + \beta').
\end{align*}
The restriction of $\alpha$ to $\Gamma(\mathcal{E})$ yields $A \in \Gamma(GL_{+}(\mathcal{E}))$ and
\begin{align*}
\Omega_{0} + \beta_{0} = A(\Omega_{0} + \beta'_{0})
\end{align*}
holds. It follows from the definition of $\mathcal{D}_{\nor}(E,\Pi)$ that $\beta'_{0}=p^{*}(\mu)$ for some $\mu \in \Gamma(E)$.
Moreover Theorem \ref{thm:normalizedMCelements} implies that $-\mu$ is coisotropic.

On the other hand given  $A \in \Gamma(GL_{+}(\mathcal{E}))$ and $\mu \in \Gamma(E)$ such that (a) and (b) are satisfied,
one can find a smooth one-parameter family
\begin{align*}
\hat{a} \in \widetilde{BFV}^{(1,1)}(E)=\Gamma(\mathcal{E}_{[0,1]}\otimes \mathcal{E}^{*}_{[0,1]})=\Gamma(\End(\mathcal{E})_{[0,1]})
\end{align*}
that generates $A \in \Gamma(GL_{+}(\mathcal{E}))$. The adjoint action of $\hat{a}$ on $BFV(E)$
with respect to $[\cdot,\cdot]_{BFV}$ integrates
to an inner automorphism $\psi$ of $BFV(E)$ according to Lemma \ref{lemma:integrability1}.
The action of $\psi^{-1}$ to $\beta$ yields a Maurer-Cartan element $\beta'$ that satisfies
$\beta'_{0}=p^{*}(\mu)$. Hence $\beta$ is in the orbit of $\mathcal{D}_{\nor}(E,\Pi)$ under the action of $\Inn_{\ge 1}(BFV(E))$.

Uniqueness of $\mu$ follows from the fact that the map which associates to $\beta$ the zero set of $\beta_{0}$
is invariant under the action of $\Inn_{\ge 1}(BFV(E))$ and $\mu$ can be reconstructed from this zero set.
\end{proof}

Let $\beta$ be a geometric Maurer-Cartan element of $(BFV(E),[\cdot,\cdot]_{BFV})$. By Lemma \ref{lemma:geometricMC}
there is a section $A \in \Gamma(GL_{+}(\mathcal{E}))$ such that
\begin{align*}
\Omega_0+\beta_0=A(\Omega_0+p^{*}(\mu))
\end{align*}
for some coisotropic section $-\mu$ of $(E,\Pi)$.

Recall that $\Gamma(\mathcal{E} \oplus \mathcal{E}^{*})$ carries a structure of symmetric pairing
induced from the natural contraction between $\mathcal{E}$ and $\mathcal{E}^{*}$. This extends
to the structure of a graded Poisson algebra on $\Gamma(\wedge \mathcal{E}\otimes \wedge \mathcal{E}^{*})$.
We denoted the corresponding graded Poisson bracket by $[\cdot,\cdot]_{G}$.

The element $\beta$ yields a differential
\begin{align*}
\delta[\beta_0](\cdot):=[\Omega_0+\beta_0,\cdot]_G
\end{align*}
on $BFV(E)=\Gamma(\wedge \mathcal{E}\otimes \wedge \mathcal{E}^{*})$.
The section $A$ acts naturally as an automorphism of $\Gamma(\wedge \mathcal{E}\otimes \wedge \mathcal{E}^{*})$
and one obtains the following commutative diagram of complexes
\begin{align*}
\xymatrix{
BFV(E) \ar[r]^{A\cdot} \ar[d]_{\delta[p^{*}(\mu)]} & BFV(E) \ar[d]^{\delta[\beta_0]} \\
BFV(E) \ar[r]^{A\cdot} & BFV(E)}
\end{align*}
which implies that the complexes $(BFV(E),\delta[\beta_0])$ and $(BFV(E),\delta[p^{*}(\mu)])$ are isomorphic. 
In particular this yields an isomorphism between the cohomologies
$H^{\bullet}(BFV(E),\delta[p^{*}(\mu)])$ and $H^{\bullet}(BFV(E),\delta[\beta_0])$.

We give more details on the computation of $H^{\bullet}(BFV(E),\delta[\beta_0])$ in order to fix a wrong statement in the proof
of Theorem 6. in \cite{Schaetz} (apologies for that). There a homotopy $h$ for $\delta[0]$
was introduced and we claimed that the operator $h$ 
is also a homotopy for $\delta[p^{*}(\mu)]$. This is not true in general, however the main line of arguments in the proof of
Theorem 6. is not effected by this.

The first step is to compute
the cohomology of $H^{\bullet}(BFV(E),\delta:=\delta[0])$. It is well-known that there are natural chain maps
\begin{align*}
i: (\Gamma(\wedge E),0) \hookrightarrow (BFV(E),\delta), \quad pr: (BFV(E),\delta) \to (\Gamma(\wedge E),0).
\end{align*}
Here $i$ is given by extending sections of $\wedge E\to S$ to sections of $\wedge \mathcal{E}\to E$ which are
constant along the fibers of $E\to S$ (recall that $\mathcal{E}\to E$ is the pull back of $E\to S$ along $E\to S$).
Moreover $pr$ is given by the projection $BFV(E)=\Gamma(\wedge \mathcal{E}\otimes \wedge \mathcal{E}^{*}) \to \Gamma(\wedge \mathcal{E})$
followed by restriction to $S$. Furthermore there is a homotopy $h:BFV(E)\to BFV(E)[-1]$ such that
\begin{align*}
h\circ \delta + \delta \circ h = \id - i\circ pr
\end{align*}
and consequently $H^{\bullet}(BFV(E),\delta)\cong \Gamma(\wedge^{\bullet}E)$.  

The next step is to compute $H^{\bullet}(BFV(E),\delta[p^{*}(\mu)])$. Associated to $\mu \in \Gamma(E)$ is a diffeomorphism $\xi[\mu]$ of the
manifold $E$ given by
\begin{align*}
\xi[\mu]: E \xrightarrow{\cong} E, \quad (x,e) \mapsto (x,e+\mu(x)).
\end{align*}
Since it maps any fiber of $E\to S$ to itself, this diffeomorphism induces an isomorphism of vector bundles
\begin{align*}
\tilde{\xi}[\mu]: \wedge \mathcal{E}\otimes \wedge \mathcal{E}^{*} \xrightarrow{\cong} \wedge \mathcal{E}\otimes \wedge \mathcal{E}^{*}
\end{align*}
covering $\xi[\mu]$. One obtains an automorphism of graded algebras
\begin{eqnarray*}
\hat{\xi}[\mu]: \Gamma(\wedge \mathcal{E}\otimes \wedge \mathcal{E}^{*}) &\xrightarrow{\cong}&\Gamma(\wedge \mathcal{E}\otimes \wedge \mathcal{E}^{*})\\
s &\mapsto& (\tilde{\xi}[\mu]^{-1} \circ s \circ \xi[\mu]) 
\end{eqnarray*}
that maps $[\cdot,\cdot]_G$ to itself and $\Omega_0$ to $\Omega_0+p^{*}(\mu)$. Consequently $\hat{\xi}[\mu]$ is an isomorphism
of chain complexes form $(BFV(E),\delta)$ to $(BFV(E),\delta[p^{*}(\mu)])$. It follows that
there are chain maps
\begin{align*}
i_{p^{*}(\mu)}: (\Gamma(\wedge E_{\mu}),0) \hookrightarrow (BFV(E),\delta), \quad \pr_{p^{*}(\mu)}: (BFV(E),\delta) \to (\Gamma(\wedge E_{\mu}),0)
\end{align*}
where $E_{\mu} \to S_{\mu}$ is the vector bundle over the graph of $-\mu$ -- which we denote by $S_{\mu}$ -- given by pulling back
$E\to S$ along the projection $S_{\mu}\xrightarrow{\cong} S$. The chain map $i_{p^{*}(\mu)}$ is given by extending sections constantly along the fibers of $E\to S$ and
$\pr_{p^{*}(\mu)}$ is given by the projection $BFV(E)=\Gamma(\wedge \mathcal{E}\otimes \wedge \mathcal{E}^{*}) \to \Gamma(\wedge \mathcal{E})$
followed by restriction to $S_{\mu}$. The homotopy $h$ for $\delta$ yields a homotopy $h[p^{*}(\mu)]: BFV(E)\to BFV(E)[-1]$ such that
\begin{align*}
h[p^{*}(\mu)]\circ \delta[p^{*}(\mu)] + \delta[p^{*}(\mu)] \circ h[p^{*}(\mu)] = \id - i_{p^{*}(\mu)}\circ \pr_{p^{*}(\mu)}
\end{align*}
holds and consequently $H^{\bullet}(BFV(E),\delta[p^{*}(\mu)])\cong \Gamma(\wedge E_{\mu})$.

The last step is to use the explicit computation of $H^{\bullet}(BFV(E),\delta[p^{*}(\mu)])$ together with the isomorphism of complexes
between $(BFV(E),\delta[\beta_0])$ and $(BFV(E),\delta[p^{*}(\mu)])$ to deduce that there are chain maps
$i_{\beta_0}$ and $\pr_{\beta_0}$ between $(\Gamma(\wedge E_{\mu}),0)$ and $(BFV(E),\delta[\beta])$. Moreover there is a homotopy
$h[\beta_0]:BFV(E)\to BFV(E)[-1]$ such that
\begin{align*}
h[\beta_0]\circ \delta[\beta_0] + \delta[\beta_0] \circ h[\beta_0] = \id - i_{\beta_0}\circ \pr_{\beta_0}
\end{align*}
holds and consequently $H^{\bullet}(BFV(E),\delta[\beta_0])\cong \Gamma(\wedge E_{\mu})$. In particular every cocycle
of $(BFV(E),\delta[\beta_0])$ that is of positive ghost-momentum degree or that vanishes when restricted to $S_{\mu}$ is a coboundary.

Lemma \ref{lemma:geometricMC} allows us to extend the map $L_{\nor}:\mathcal{D}_{\nor}(E,\Pi) \to \Gamma(E)$ to
\begin{align*}
L_{\geo}: \mathcal{D}_{\geo}(E,\Pi) \to \Gamma(E), \quad \beta \mapsto -\mu_{\beta}.
\end{align*}

Theorem \ref{thm:normalizedMCelements} extends in a straightforward fashion to

\begin{Theorem}\label{thm:geometricMCelements}
The map $L_{\geo}$ has the following properties:
\begin{itemize}
\item[(a)] it maps onto the set of coisotropic sections $\mathcal{C}(E,\Pi) \subset \Gamma(\mathcal{E})$
(see Definition \ref{def:coisotropicsections} in Subsection \ref{ss:coisotropicsections}),
\item[(b)] it is invariant under the action of $\Inn_{\ge 1}(BFV(E))$ on $\mathcal{D}_{\geo}(E,\Pi)$,
\item[(c)] it induces an isomorphism
\begin{align*}
[L]: \mathcal{D}_{\geo}(E,\Pi) / \Inn_{\ge 1}(BFV(E)) \xrightarrow{\cong} \mathcal{C}(E,\Pi).
\end{align*}
\end{itemize}
\end{Theorem}

At this stage the purpose of Theorem \ref{thm:geometricMCelements} is not clear. However 
it will turn out that the appropriate starting point for the construction of the geometric BFV-groupoid is not
$\mathcal{D}_{\nor}(E,\Pi)$ with its action of $\Inn_{\ge 2}(BFV(E))$ but
$\mathcal{D}_{\geo}(E,\Pi)$ with its action of $\Inn_{\ge 1}(BFV(E))$.

\subsection{Gauge homotopies}\label{ss:gaugehomotopies}
The set of geometric Maurer-Cartan elements $\mathcal{D}_{\geo}(E,\Pi)$ is a subset of
$\mathcal{D}_{\alg}(E,\Pi)$. The latter set is acted upon by the group of inner automorphisms
$\Inn(BFV(E))$ and its subgroups $\Inn_{\ge r}(BFV(E))$.

\begin{Definition}\label{def:gaugehomotopies}
Let $\beta_0$, $\beta_1$ be elements of $\mathcal{D}_{\geo}(E,\Pi)$. A {\em gauge homotopy}
from $\beta_0$ to $\beta_1$ is a pair $(\hat{\beta},\hat{\psi})$ where
\begin{itemize}
\item[(a)] $\hat{\beta}$ is an element of $\widetilde{BFV}^{1}(E)$ such that
its restriction to $E\times \{t\}$ lies in $\mathcal{D}_{\geo}(E,\Pi)$ for arbitrary $t\in [0,1]$ and
\item[(b)] $\hat{\psi}$ is smooth one-parameter family of inner automorphisms, i.e. an element of
$\underline{\Inn}(BFV(E))$
\end{itemize}
such that
\begin{itemize}
\item[(a')] the restriction of $\hat{\beta}$ to $E\times \{0\}$ is $\beta_0$ and the restriction to $E\times \{1\}$ is
$\beta_1$,
\item[(b')] $\beta_t=\psi_{t}(\beta_0)$ holds for all $t \in [0,1]$.
\end{itemize}
A gauge homotopy is called {\em pure} if the associated smooth one-parameter family of inner automorphisms $\hat{\psi}$
is an element of $\underline{\Inn}_{\ge 1}(BFV(E))$.
\end{Definition}
 
We remark that we allow for arbitrary smooth families of inner automorphisms $\hat{\psi}$ of $(BFV(E),[\cdot,\cdot]_{BFV})$
to appear as part of the data for a gauge homotopy, not only those which lie in $\underline{\Inn}_{\ge 1}(BFV(E))$ and
automatically map any geometric Maurer-Cartan element to a geometric Maurer-Cartan element.
Condition (b') in the definition of gauge homotopies 
essentially says that the ``time-dependent vector field'' generating
the gauge homotopy is tangential to $\mathcal{D}_{\geo}(E,\Pi)$ along the path $\hat{\beta}$ contained
in the subset $\mathcal{D}_{\geo}(E,\Pi)$.

\begin{Lemma}\label{lemma:gaugehomotopy}
The relation on $\mathcal{D}_{\geo}(E,\Pi)$ given by
\begin{align*}
\left(\beta \sim_G \delta \right) :\Leftrightarrow \quad \text{there is a gauge homotopy from } \beta \text{ to } \delta
\end{align*}
is an equivalence relation.
\end{Lemma}

The proof is essentially the same as the proof of Lemma \ref{lemma:Hamiltonianhomotopy} in Subsection \ref{ss:Hamiltonianhomotopies}.
In particular every geometric Maurer-Cartan element comes along with a fixed gauge homotopy $\id_{\beta}$ and
for every gauge homotopy $(\hat{\beta},\hat{\psi})$ from $\beta_0$ to $\beta_1$ there is an associated one
from $\beta_1$ to $\beta_0$ which we denote by $(\hat{\beta},\hat{\psi})^{-1}$.
Furthermore every choice of a gluing function $\rho$ defines an operation $\Box_{\rho}$ which maps two gauge homotopies
$(\hat{\alpha},\hat{\phi})$ and $(\hat{\beta},\hat{\psi})$ 
with matching data attached to the boundary points $\{1\}$ and $\{0\}$ respectively to a new gauge homotopy
denoted by $(\hat{\beta},\hat{\psi})\Box_{\rho}(\hat{\alpha},\hat{\psi})$.

\begin{Definition}\label{def:moduliofgeometricMCelements}
Let $(E,\Pi)$ be a coisotropic vector bundle.
We denote the set of equivalence classes in $\mathcal{D}_{\geo}(E,\Pi)$ with respect to $\sim_{G}$
by $\mathcal{N}(E,\Pi)$ and call it the set of {\em geometric Maurer-Cartan elements} of $(E,\Pi)$ {\em modulo gauge homotopies}
or the {\em moduli space of geometric Maurer-Cartan elements} of $(E,\Pi)$.
\end{Definition}

\subsection{The geometric BFV-groupoid}\label{ss:geometricBFV-groupoid}

We want to construct a groupoid whose set of orbits is equal to the moduli space of geometric Maurer-Cartan elements 
$\mathcal{N}(E,\Pi)$. The main problem is that there is no ``natural'' composition on the set of gauge homotopies
(with matching data at the end of the first one and at the beginning of the second one respectively). The operation
$\Box_{\rho}$ depends on the choice of a gluing function $\rho$ and is not associative. To overcome this problems we
introduce an equivalence relation on the set of gauge homotopies:

\begin{Definition}\label{def:isotopiesofgaugehomotopies}
Let $(E,\Pi)$ be a coisotropic vector bundle.
An {\em isotopy of gauge homotopies} is a pair $(\hat{\beta},\hat{\Psi})$ where
\begin{itemize}
\item[(a)] $\hat{\beta}$ is a section of the pull back $\wedge \mathcal{E}_{[0,1]^{2}}\otimes \wedge \mathcal{E}^{*}_{[0,1]^{2}}$
 of $\wedge \mathcal{E}\otimes \wedge \mathcal{E}^{*}$ along 
$E\times [0,1]^{2}\to E$ such that its restriction to $E\times \{t\} \times \{s\}$ lies in
$\mathcal{D}_{\geo}(E,\Pi)$ for arbitrary $t,s\in [0,1]$ and 
\item[(b)] $\hat{\Psi}: BFV(E) \to \Gamma(\wedge \mathcal{E}_{[0,1]^{2}}\otimes \wedge \mathcal{E}^{*}_{[0,1]^{2}})$
is a morphism of unital graded Poisson algebras whose composition with the evaluation at $E\times \{t\}\times \{s\}$ 
for $t,s\in [0,1]$ arbitrary is
an automorphism of unital graded Poisson algebras 
\end{itemize}
such that
\begin{itemize}
\item[(a')] the composition of $\hat{\Psi}$ with the evaluation at $E\times \{0\}\times [0,1]$ is the identity,
\item[(b')] the restriction of $\hat{\beta}$ to $E\times \{1\}\times [0,1]$ is constant in $s\in [0,1]$, 
\item[(c')] there is a section $\hat{\gamma}$ of the bundle
\begin{align*}
\oplus_{m\ge 0}\left(\wedge^{m} \mathcal{E}_{[0,1]^{2}}\otimes \wedge^{m} \mathcal{E}^{*}_{[0,1]^{2}} \right)
\end{align*}
such that the composition of $\hat{\Psi}$ with the evaluation at $E\times [0,1] \times \{s\}$ for arbitrary $s\in [0,1]$
is the gauge homotopy generated by
the restriction of $\hat{\gamma}$ to $E\times I \times \{s\}$ and
\item[(d')] the image of $\beta_{0,s}$ under the composition of $\hat{\Psi}$
with evaluation at $E\times \{t\} \times \{s\}$ is equal to $\beta_{t,s}$ for all $s,t \in [0,1]$.
\end{itemize}

We say that an isotopy of gauge homotopies $(\hat{\beta},\hat{\Psi})$ starts at the gauge homotopy
$(\hat{\beta}|_{E\times [0,1] \times \{0\}},ev_{s=0}\circ\hat{\Psi})$ and ends at the gauge homotopy\\
$(\hat{\beta}|_{E\times [0,1] \times \{1\}},ev_{s=1}\circ\hat{\Psi})$.
\end{Definition}

\begin{Lemma}\label{lemma:isotopiesofgaugehomotopies}
\begin{itemize}
\item[(a)]
The relation on the set of gauge homotopies given by
\begin{align*}
(\hat{\alpha},\hat{\phi}) \simeq_{G} (\hat{\beta},\hat{\psi}) :\Leftrightarrow
\end{align*}
there is an isotopy of gauge homotopies from $(\hat{\alpha},\hat{\phi})$ to $(\hat{\beta},\hat{\psi})$;\\
defines an equivalence relation on the set of gauge homotopies.
\item[(b)] Let $\rho$ and $\rho'$ be two gluing functions. Then the compositions
of gauge homotopies with respect to $\rho$ and $\rho'$ coincide
up to $\simeq_{G}$.
\item[(c)] The gauge homotopies
\begin{align*}
\id_{\alpha_0}\Box_{\rho}(\hat{\alpha},\hat{\phi}) \text{ and } (\hat{\alpha},\hat{\phi})\Box_{\rho}\id_{\alpha_1}
\end{align*}
are equivalent to $(\hat{\alpha},\hat{\phi})$.
\item[(d)]
The Hamiltonian homotopies 
\begin{align*}
(\hat{\alpha},\hat{\phi})^{-1}\Box_{\rho}(\hat{\alpha},\hat{\phi}) \text{ and } (\hat{\alpha},\hat{\phi})\Box_{\rho}(\hat{\alpha},\hat{\phi})^{-1}
\end{align*}
are equivalent to $\id_{\alpha_{0}}$.
\item[(e)]
The operation $\Box_{\rho}$
descends to the set of gauge homotopies modulo isotopies of Hamiltonian homotopies and is associative there.
\end{itemize}
\end{Lemma}

The proof can be copied mutatis mutandis from the proof of Lemma \ref{lemma:isotopiesofHamiltonianhomotopies}
in Subsection \ref{ss:groupoidofcoisotropicsections}.

\begin{Definition}\label{def:geometricBFV-groupoid}
The {\em geometric BFV-groupoid} $\hat{\mathcal{D}}_{\geo}(E,\Pi)$ associated to a coisotropic vector bundle $(E,\Pi)$
is the small groupoid where
\begin{itemize}
\item[(a)] the set of objects is the set $\mathcal{D}_{\geo}(E,\Pi)$ of all geometric Maurer-Cartan elements
of $(BFV(E),[\Omega,\cdot]_{BFV},[\cdot,\cdot]_{BFV})$,
\item[(b)] the set of morphisms $\Hom(\beta,\delta)$ between two geometric Maurer-Cartan elements $\beta$
and $\delta$ is the set of all gauge homotopies form $\beta$ to $\delta$ modulo isotopies of gauge homotopies and
\item[(c)] the composition is induced from coposition of gauge homotopies with respect to some gluing function.
\end{itemize}
\end{Definition}

The geometric BFV-groupoid $\hat{\mathcal{D}}_{\geo}(E,\Pi)$ can be seen as the restriction of a
groupoid of Maurer-Cartan elements $\hat{\mathcal{D}}_{\alg}(E,\Pi)$ associated to the differential graded Poisson algebra
$(BFV(E),[\Omega,\cdot]_{BFV},[\cdot,\cdot]_{BFV})$ with morphisms given
by gauge homotopies modulo isotopies. It seems very likely that the groupoid $\hat{\mathcal{D}}_{\geo}(E,\Pi)$ can be understood
as the truncation of a weak $\infty$-groupoid $\hat{\mathcal{D}}^{\infty}_{\geo}(E,\Pi)$ at its two-morphisms which should be given by
isotopies of gauge homotopies.
The set of orbits of $\hat{\mathcal{D}}_{\geo}(E,\Pi)$ is the moduli space $\mathcal{N}(E,\Pi)$ of geometric Maurer-Cartan elements
of $(E,\Pi)$ modulo gauge homotopies.

\subsection{The BFV-groupoid}\label{ss:BFV-groupoid}

\begin{Definition}\label{def:pureclasses}
A morphism in the groupoid $\hat{\mathcal{D}}_{\geo}(E,\Pi)$ is called {\em pure} if there is a pure gauge homotopy
(Definition \ref{def:gaugehomotopies}) representing it.

We denote the set of pure morphisms between two geometric Maurer-Cartan elements $\beta$ and $\delta$ of $(E,\Pi)$ by
$Hom_{\ge 1}(\beta,\delta)\subset \Hom(\beta,\delta)$.
\end{Definition}

\begin{Definition}
Let $\mathcal{G}$ be a small groupoid. 

A subgroupoid $\mathcal{H}$ is
\begin{itemize}
\item[(a)] {\em full} if every object in $\mathcal{G}$ is an object in $\mathcal{H}$.
\item[(b)] {\em normal} if for every morphism $f$ from $X$ to $Y$ in $\mathcal{G}$ and every morphism $g$ from $Y$ to $Y$ in $\mathcal{H}$,
$f^{-1}\circ g \circ f$ is a morphism in $\mathcal{H}$.
\end{itemize}
\end{Definition}

\begin{Definition}\label{def:quotientgroupoid}
Let $\mathcal{G}$ be a small groupoid and $\mathcal{H}$ a full normal subgroupoid of $\mathcal{G}$.
Then the {\em quotient} of $\mathcal{G}$ by $\mathcal{H}$ is the groupoid where
\begin{itemize}
\item[(a')] objects $[X]$ are equivalence classes of objects of $\mathcal{G}$ with the relation $[X] = [Y]$ if $\Hom_{\mathcal{H}}(X,Y)\neq \emptyset$.
\item[(b')] morphisms $[\alpha]$ are equivalence classes of morphisms of $\mathcal{G}$ with the relation $[\alpha] = [\beta]$ for
$\alpha \in \Hom_{\mathcal{G}}(X,Y)$, $\beta \in \Hom_{\mathcal{G}}(W,Z)$ if there are $f\in \Hom_{\mathcal{H}}(X,W)$ and $g \in \Hom_{\mathcal{H}}(Y,Z)$ such that $g\circ \alpha =\beta \circ f$ holds.
\end{itemize}
\end{Definition}

\begin{Theorem}\label{thm:normalsubgroupoid}
The class of pure morphisms in $\hat{\mathcal{D}}_{\geo}(E,\Pi)$ yields a full normal subgroupoid of $\hat{\mathcal{D}}_{\geo}(E,\Pi)$.
\end{Theorem}

\begin{proof}
The crucial point is normality of this full subgroupoid. This is an immediate consequence of Proposition \ref{prop:kernel}
in Subsection \ref{ss:morphism}.
\end{proof}

\begin{Definition}\label{def:BFV-groupoid}
The {\em BFV-groupoid} $\hat{\mathcal{D}}(E,\Pi)$ associated to a coisotropic vector bundle $(E,\Pi)$
is the quotient of $\hat{\mathcal{D}}_{\geo}(E,\Pi)$ by the class of pure morphisms. We denote its
set of objects by $\mathcal{D}(E,\Pi)$.
\end{Definition}

\begin{Lemma}\label{lemma:objectsofBFV-groupoid}
The set of objects $\mathcal{D}(E,\Pi)$ of $\hat{\mathcal{D}}(E,\Pi)$ is the set of orbits of the action of 
$\Inn_{\ge 1}(BFV(E))$ on $\mathcal{D}_{\geo}(E,\Pi)$.
\end{Lemma}

\begin{proof}
By definition the set of objects $\mathcal{D}(E,\Pi)$ is the set of equivalence classes in $\mathcal{D}_{\geo}(E,\Pi)$
under the equivalence relation $(X \sim Y :\Leftrightarrow $ there is a pure morphism from $X$ to $Y)$. The existence
of a pure morphism from $X$ to $Y$ is equivalent to the existence of a pure gauge homotopy from $X$ to $Y$. This in turn
is equivalent to the existence of an element of $\Inn_{\ge 1}(BFV(E))$ which maps $X$ to $Y$.
\end{proof}

By Theorem \ref{thm:geometricMCelements} there is an isomorphism of sets
\begin{align*}
[L]: \mathcal{D}(E,\Pi)= \mathcal{D}_{\geo}(E,\Pi)/\Inn_{\ge 1}(BFV(E)) \xrightarrow{\cong} \mathcal{C}(E,\Pi)
\end{align*}
from the set of objects of $\hat{\mathcal{D}}(E,\Pi)$ to the set of objects of $\hat{\mathcal{C}}(E,\Pi)$.
In the next Section this isomorphism is extended to an isomorphism of groupoids
$\mathcal{L}:\hat{\mathcal{D}}(E,\Pi) \xrightarrow{\cong} \hat{\mathcal{C}}(E,\Pi)$.

\section{Isomorphism of the two deformation groupoids}\label{s:isomorphism}

Let $(E,\Pi)$ be a coisotropic vector bundle.
In Theorem \ref{thm:main} we prove that the groupoid of coisotropic sections $\hat{\mathcal{C}}(E,\Pi)$ (introduced in Subsection
\ref{ss:groupoidofcoisotropicsections}) is isomorphic to the BFV-groupoid $\hat{\mathcal{D}}(E,\Pi)$ (introduced in
Subsection \ref{ss:BFV-groupoid}). As a corollary we obtain an isomorphism between the moduli space
of coisotropic sections $\mathcal{M}(E,\Pi)$ of $(E,\Pi)$ (Subsection \ref{ss:Hamiltonianhomotopies})
and the moduli space of geometric Maurer--Cartan elements $\mathcal{N}(E,\Pi)$ (Subsection \ref{ss:gaugehomotopies}).
Although this chain of arguments would be pleasing from a conceptual point of view, it is technically cumbersome. In particular
the verification of Theorem \ref{thm:normalsubgroupoid} Subsection \ref{ss:geometricBFV-groupoid} poses problems.

So instead we will take another route and first prove the result concerning the moduli spaces in Theorem \ref{thm:equivalenceclasses}
independently.
Then we extend the isomorphism between
the moduli spaces to a morphism of groupoids from $\hat{\mathcal{D}}_{\geo}(E,\Pi)$ to $\hat{\mathcal{C}}(E,\Pi)$.
Proposition \ref{prop:kernel} assures that the kernel of this morphism is equal to the kernel
of $\hat{\mathcal{D}}_{\geo}(E,\Pi)\to \hat{\mathcal{D}}(E,\Pi)$. This implies Theorem \ref{thm:normalsubgroupoid} and yields
the isomorphism
between the groupoids $\hat{\mathcal{D}}(E,\Pi)$ and $\hat{\mathcal{C}}(E,\Pi)$.

\subsection{Relating the inner automorphisms}\label{ss:innerautorphisms}

Consider $\hat{\psi} \in \underline{\Inn}(BFV(E))$, i.e. a smooth one-parameter family of inner automorphisms of
the  differential graded Poisson algebra $(BFV(E),[\cdot,\cdot]_{BFV})$. Let $\hat{\gamma}$ be an element of 
$\widetilde{BFV}^{0}(E)=\oplus_{m\ge 0}\Gamma(\wedge^{m}\mathcal{E}_{[0,1]}\otimes \wedge^{m}\mathcal{E}_{[0,1]}^{*})$ that
generates $\hat{\psi}$.
We denote its component in $\widetilde{BFV}^{(0,0)}(E)=\mathcal{C}^{\infty}(E\times [0,1])$ by $\pi(\hat{\gamma})$.

The family $\hat{\psi}$ induces a smooth one-parameter family of Poisson automorphisms $[\hat{\psi}]$ of
$(\mathcal{C}^{\infty}(E),\{\cdot,\cdot\}_{\Pi})$ via
\begin{align*}
\mathcal{C}^{\infty}(E)\hookrightarrow BFV^{0}(E) \xrightarrow{\psi_t} BFV^{0}(E) \xrightarrow{\pi} BFV^{(0,0)}(E)= \mathcal{C}^{\infty}(E).
\end{align*}
It is known that any Poisson automorphism is equl to the pull-back by some diffeomorphism of $E$ which is a Poisson map.
To find out which particular one-parameter family of Poisson-diffeomorphisms corresponds to $\hat{\psi}$ we compute
\begin{eqnarray*}
\frac{d}{dt}|_{t=s}(\pi\circ \psi_t)(f)&=&\pi\left(\frac{d}{dt}|_{t=s}(\psi_t)(f) \right)=\pi\left(-[\gamma_s,\psi_t(f)]_{BFV}\right)\\
&=&-\{\pi(\gamma_{s}),(\pi\circ \psi_t)(f)\}_{\Pi}.
\end{eqnarray*}
This implies that $\pi\circ\hat{\psi}$ is equal to $\left((\phi_t^{-1})^{*}\right)_{t\in [0,1]}$ where
$(\phi_t)_{t\in I}$ is the smooth one-parameter family of Hamiltonian diffeomorphisms satisfying
\begin{align*}
\frac{d}{dt}|_{t=s}\phi_t=X_{\pi(\gamma_s)}|_{\phi_s}, \quad \phi_0=\id.
\end{align*}

\begin{Lemma}\label{lemma:integrability2}
Given $\hat{\gamma} \in \widetilde{BFV}^{0}(E)$ there is a (necessarily unique) element $\hat{\psi}$ of
$\underline{\Inn}(BFV(E))$ satisfying
\begin{align*}
\frac{d}{dt}|_{t=s}\psi_t=-[\gamma_s,\psi_s]_{BFV}, \quad \psi_0=\id
\end{align*}
if and only if $\pi(\hat{\gamma})$ integrates to a
smooth one-parameter family $\hat{\phi}$ of Hamiltonian diffeomorphisms of $(E,\Pi)$.
\end{Lemma}

\begin{proof}
From the remarks made above it is straightforward to check that the implication ($\Rightarrow$) holds.

For the converse we assume that $F:=\pi(\hat{\gamma}): E\times [0,1] \to \mathbb{R}$ can
be integrated to a smooth one-parameter family of Hamiltonian diffeomorphisms $(\phi_t)_{t\in [0,1]}$
of $(E,\Pi)$.

We fixed a connection $\nabla$ on $E\to S$ in Subsection \ref{ss:BFV-complex} and used it to construct the
BFV-bracket $[\cdot,\cdot]_{BFV}$ on $BFV(E)$.
The connection $\nabla$ on $E\to S$ induces a connection on $\mathcal{E}\to E$ via pull back. Using parallel transport with respect
to this connection one lifts
$(\phi_t)_{t\in [0,1]}$ to a family of vector bundle isomorphisms 
\begin{align*}
\tilde{\phi}_t: \wedge \mathcal{E} \otimes \wedge \mathcal{E}^{*} \xrightarrow{\cong} \wedge \mathcal{E} \otimes \wedge \mathcal{E}^{*}
\end{align*}
covering $(\phi_t)_{t\in [0,1]}$. This induces a morphism of unital algebras
\begin{align*}
\hat{\Phi}: \Gamma(\wedge \mathcal{E}\otimes \wedge \mathcal{E}^{*}) \to
 \Gamma(\wedge \mathcal{E}_{[0,1]}\otimes \wedge \mathcal{E}^{*}_{[0,1]}),\quad
 \beta \mapsto (\tilde{\phi}_t) \circ \beta \circ \phi_t^{-1}
\end{align*}
which generalizes the push forward 
\begin{align*}
\mathcal{C}^{\infty}(E) \to \mathcal{C}^{\infty}(E\times [0,1]), \quad f \mapsto f\circ (\phi_t)^{-1}.
\end{align*}

One checks that
\begin{align*}
\frac{d}{dt}|_{t=s}\Phi_{t}=-\nabla_{X_{\pi(\gamma_s)}}\circ \Phi_s, \quad \Phi_0=\id
\end{align*}
holds where $\nabla_{(\cdot)}$ denotes the covariant derivative of $\Gamma(\wedge \mathcal{E}\otimes \wedge \mathcal{E}^{*})$ with
respect to the connection induced by $\nabla$ and $X_{\pi(\gamma_s)}$ is the Hamiltonian vector field of $\pi(\gamma_s)$ on $(E,\Pi)$.
Consequently $\hat{\Phi}$ is the smooth one-parameter family of automorphism of
$\Gamma(\wedge \mathcal{E}\otimes \wedge \mathcal{E}^{*})$ which integrates the smooth one-parameter family of derivations
$(-\nabla_{X_{\pi(\gamma_t)}})_{t\in [0,1]}$.

The flow equation for $-[\hat{\gamma},\cdot]_{BFV}$ is equivalent to
\begin{align*}
\frac{d}{dt}|_{t=s}\varphi_t=\left(\Phi_s^{-1}\circ(\nabla_{X_{\pi(\gamma_s)}}-[\gamma_s,\cdot]_{BFV})\circ \Phi_s\right)(\varphi_s)
\end{align*}
where $\varphi_t:=\Phi_t^{-1}\circ \psi_t$.
The derivation $[\gamma_s,\cdot]_{BFV}$ can be decomposed as $\nabla_{X_{\pi(\gamma_s)}}+[\gamma^{1}_s,\cdot]_{G}$ where
$\gamma^{1}_s$ is the component of $\gamma_s$ in $BFV^{(1,1)}(E)$
plus a
nilpotent part. Hence
\begin{eqnarray*}
\Phi_s^{-1}(\nabla_{X_{\pi(\gamma_s)}}-[\gamma_s,\cdot]_{BFV})\Phi_s&=&
\Phi_s^{-1}(-[\gamma^{1}_s,\cdot]_{G} + \text{nilpotent part})\Phi_s\\
&=&[-\Phi_s^{-1}\gamma^{1}_s,\cdot]_{G} + \text{nilpotent part}
\end{eqnarray*}
and in Lemma \ref{lemma:integrability1} a derivation of that form was proved to admit a unique flow.
\end{proof}

Consequently we obtain a map 
\begin{align*}
\underline{L}:\underline{\Inn}(BFV(E)) \to \underline{\Ham}(E,\Pi)
\end{align*} 
given by mapping the flow generated
by $\hat{\gamma}\in \widetilde{BFV}^{0}(E,\Pi)$ to the flow generated by its projection $\pi(\hat{\gamma})$
to $\mathcal{C}^{\infty}(E\times [0,1])$. Furthermore there is a map
\begin{align*}
\underline{R}: \underline{\Ham}(E,\Pi) \to \underline{\Inn}(BFV(E))
\end{align*} 
given by
mapping the flow generated by $F\in \mathcal{C}^{\infty}(E\times [0,1])$ to the flow generated by
$F\in \mathcal{C}^{\infty}(E\times [0,1]) = \widetilde{BFV}^{(0,0)}(E)\subset \widetilde{BFV}^{0}(E)$.

Clearly $\underline{L}\circ \underline{R}=\id$ and so $\underline{L}$ is surjective and $\underline{R}$ is injective.
Moreover $\underline{L}$ and $\underline{R}$ are homomorphism with respect to the
group structures on $\underline{\Inn}(BFV(E))$ and $\underline{\Ham}(E,\Pi)$ respectively. The kernel of $\underline{L}$ obviously
contains
the subgroup $\underline{\Inn}_{\ge 1}(BFV(E))$ of $\underline{\Inn}(BFV(E))$.

\begin{Lemma}\label{lemma:kernel0}
The kernel of $\underline{L}: \underline{\Inn}(BFV(E))\to \underline{\Ham}(E,\Pi)$ is the subgroup
$\underline{\Inn}_{\ge 1}(BFV(E))$.
\end{Lemma}

\begin{proof}
Consider $\hat{\phi}\in \underline{\Inn}(BFV(E))$ with $\underline{L}(\hat{\phi})=\id$. Assume
$\hat{\phi}$ is generated by $\hat{\gamma}\in \Gamma(\wedge \mathcal{E}_{[0,1]}\otimes \wedge \mathcal{E}_{[0,1]}^{*})$.
We decompose $\hat{\gamma}$ with respect to the filtration $\Gamma(\wedge^{m}\mathcal{E}_{[0,1]}\otimes \wedge^{m}\mathcal{E}^{*}_{[0,1]})$
\begin{align*}
\hat{\gamma}=\hat{\gamma}^{0}+\hat{\gamma}^{1}+\hat{\gamma}^{2}+\cdots
\end{align*}
The identity $\underline{L}(\hat{\phi})=\id$ implies that the Hamiltonian vector field $X_{F}$ associated to the function
$F:=\hat{\gamma}^{0}:E\times [0,1]\to \mathbb{R}$ vanishes, i.e. $<\Pi,dF>=0$.

If $[F,\cdot]_{BFV}=0$ would hold
\begin{eqnarray*}
[\hat{\gamma},\cdot]_{BFV}&=&[F,\cdot]_{BFV}+[\hat{\gamma}^{1}+\hat{\gamma}^{2}+\cdots,\cdot]_{BFV}\\
&=& [\hat{\gamma}^{1}+\hat{\gamma}^{2}+\cdots,\cdot]_{BFV}
\end{eqnarray*}
and so $\hat{\phi}$ would be generated by the element $\hat{\gamma}^{1}+\hat{\gamma}^{2}+\cdots$, i.e.
$\hat{\phi}\in \underline{\Inn}_{\ge 1}(BFV(E))$.

The first contribution to $[F,\cdot]_{BFV}$ is given by $[F,\cdot]_{G}$ where $[\cdot,\cdot]_{G}$ encodes
the fiber pairing between $\mathcal{E}$ and $\mathcal{E}^{*}$. Consequently $[F,\cdot]_{G}=0$.
The next contribution is $\nabla_{X_{F}}=0$. Following the explicit construction of $[\cdot,\cdot]_{G}$
in \cite{Schaetz} one finds that all higher contributions to $[\cdot,\cdot]_{BFV}$ can be written in terms of
wedge products of the horizontal lift $\iota_{\nabla}(\Pi)$ of $\Pi$ with respect to the fixed connection $\nabla$ and contraction
with the curvature tensor $R_{\nabla}\in \Omega^{2}(E,\End(\mathcal{E}))$ interpreted as an element of
$\Omega^{2}(\mathcal{E},\End(\mathcal{E}))$ via pull-back. Hence if we contract one of these terms with $dF$ we obtain
contributions proportional to $<\iota_{\nabla}(\Pi),dF>=\iota_{\nabla}(X_{F})=0$.
\end{proof}

\subsection{An isomorphism of moduli spaces}\label{ss:isomorphism_moduli}
In Subsection \ref{ss:Hamiltonianhomotopies} the moduli space of coisotropic sections $\mathcal{M}(E,\Pi)$
of a coisotropic vector bundle $(E,\Pi)$ was introduced. It is the set of equivalence classes of coisotropic sections
$\mu \in \mathcal{C}(E,\Pi)$ under the equivalence relation $\sim_{H}$ given by Hamiltonian homotopies.

On the other hand we introduced the moduli space $\mathcal{N}(E,\Pi)$ of geometric Maurer--Cartan elements
of $(E,\Pi)$ in Subsection \ref{ss:gaugehomotopies}. Recall that $\mathcal{N}(E,\Pi)$ is
the set of equivalence classes of geometric Maurer--Cartan elements $\beta \in \mathcal{D}_{\geo}(E,\Pi)$
modulo the equivalence relation $\sim_{G}$ given by gauge homotopies.

In Subsection \ref{ss:geometricMC-elements} a surjective map
\begin{align*}
L_{\geo}: \mathcal{D}_{\geo}(E,\Pi) \twoheadrightarrow \mathcal{C}(E,\Pi)
\end{align*}
from the set of geometric Maurer--Cartan elements $\mathcal{D}_{\geo}(E,\Pi)$ of $(E,\Pi)$ to the set of
coisotropic sections $\mathcal{C}(E,\Pi)$ was introduced.

\begin{Theorem}\label{thm:equivalenceclasses}
Let $(E,\Pi)$ be a coisotropic vector bundle. Then the map
\begin{align*}
L_{\geo}: \mathcal{D}_{\geo}(E,\Pi) \twoheadrightarrow \mathcal{C}(E,\Pi)
\end{align*}
induces a bijection
\begin{align*}
[L_{\geo}]: \mathcal{N}(E,\Pi) \xrightarrow{\cong} \mathcal{M}(E,\Pi).
\end{align*}
\end{Theorem}

\begin{proof}
Let $(\hat{\alpha},\hat{\psi})$ be a gauge homotopy from the geometric Maurer--Cartan element $\alpha$ to 
the geometric Maurer--Cartan element $\beta$. Then
\begin{align*}
\hat{L}_{\geo}(\hat{\alpha},\hat{\psi}):=(L_{\geo}(\hat{\alpha}),\underline{L}(\hat{\psi}))
\end{align*}
is a Hamiltonian homotopy from $L_{\geo}(\alpha)$ to $L_{\geo}(\beta)$. Hence
$L_{\geo}$ factorizes to a map from $\mathcal{N}(E,\Pi)$ to $\mathcal{M}(E,\Pi)$ which we denote by $[L_{\geo}]$.
Since $L_{\geo}$ is surjective, so is $[L_{\geo}]$.

To prove injectivity we have to show that $L_{\geo}(\alpha)\sim_{H}L_{\geo}(\beta)$ implies $\alpha \sim_{G} \beta$.
We set $-\mu:=L_{\geo}(\alpha)$ and $-\nu:=L_{\geo}(\beta)$ and choose a Hamiltonian homotopy $(-\hat{\mu},\hat{\phi})$
from $-\mu$ to $-\nu$.

Consider the smooth one-parameter family of coisotropic sections $-\hat{\mu}$. In \cite{Schaetz}
it was proved that every coisotropic section $-\mu_t$ can be extented to a normalized (and hence geometric) Maurer--Cartan element
$\gamma_t$ of $BFV(E)$. One way to construct $\gamma_t$
uses the complex $(BFV(E),\delta[p_{!}(\mu_t)])$ and the associated
homotopy $h[p_{!}(\mu_t)]$. The extension of $\Omega_0+p_{!}(\mu_t)$ to a geometric Maurer--Cartan element 
is constructed in an iterative procedure where
the vanishing of certain obstruction classes in $H^{\bullet}(BFV(E),\delta[p_{!}(\mu_t)])$ is proven at every step. In order to find
cochains in $BFV(E)$ whose images under $\delta[p_{!}(\mu_t)]$ cancel the obstruction elements one uses the homotopy
$h[p_{!}(\mu_t)]$. Since $h[p_{!}(\mu_t)]$ depends smoothly on $t\in [0,1]$ so does the constructed family of geometric Maurer-Cartan
elements $\hat{\gamma}:=(\gamma_t)_{t\in [0,1]}$.

In the previous Subsection it was shown that the smooth one-parameter family of Hamiltonian
diffeomorphisms $\hat{\phi}$ can be lifted to a smooth one-parameter family of inner automorphisms $\hat{\varphi}$
of $(BFV(E),[\cdot,\cdot]_{BFV})$. Consider the smooth one-parameter family of Maurer--Cartan elements
\begin{align*}
\hat{\delta}:=(\varphi_t^{-1}\cdot \gamma_t)_{t\in [0,1]}.
\end{align*}
We claim that there is an element $\hat{\Phi} \in \underline{\Inn}_{\ge 1}(BFV(E))$ such that
$(\hat{\varphi}\circ\hat{\Phi})\cdot \delta_0$ is geometric. Consequently
$((\hat{\varphi}\circ \hat{\Phi})\cdot \delta_0,\hat{\varphi}\circ\hat{\Phi})$ is a gauge homotopy from $\gamma_0$ to
$(\varphi_1\circ \hat{\Phi})\cdot \gamma_0$.

So consider the smooth one-parameter family of Maurer--Cartan elements
$\hat{\delta}$. Denote the component of $\delta_t$ in $BFV^{(1,0)}(E)=\Gamma(\mathcal{E})$ by $\sigma_t$.
The section $\Omega_0+\sigma_t$ intersects the zero section of $\mathcal{E}\to E$
in the graph of $\mu_0$ for arbitrary $t\in [0,1]$.
Moreover this intersection can be checked to be transversal.

Since $\delta_0$ is geometric we have detailed information on the cohomology of the complex $(BFV(E),\delta[\sigma_0])$,
see Subsection \ref{ss:geometricMC-elements}. 
We make the following observations:
\begin{itemize}
\item[(a')] $\delta[\sigma_0](\Omega_0+\sigma_t)=0$ for arbitrary $t \in [0,1]$ since
the differential is defined via the pairing between $\mathcal{E}$ and $\mathcal{E}^{*}$,
\item[(b')] $\pr_{\sigma_0}(\Omega_0+\sigma_t)=0$ for arbitrary $t\in [0,1]$ since 
the projection involves evaluation of the section
$\Omega_0+\sigma_t$ on the vanishing locus of $\Omega_0+\sigma_0$ and $\Omega_0+\sigma_t$ vanishes there as well,
\item[(c')] hence the formula for the homotopy $h[\alpha_0]$ implies
\begin{align*}
\delta[\sigma_0](h[\sigma_0](\Omega_0+\sigma_t))=\Omega_0+\sigma_t
\end{align*}
for arbitrary $t\in [0,1]$.
\end{itemize}
Define $M_t:=h[\sigma_0](\Omega_0+\sigma_t) \in \Gamma(\mathcal{E}\otimes \mathcal{E}^{*})$ which we interpret as a smooth
family of sections of $\Gamma(\End(\mathcal{E}))$ parameterized by $[0,1]$. The identity in (c') translates into
\begin{align*}
M_t(\Omega_0+\sigma_0)=\Omega_0+\sigma_t
\end{align*}
and $M_0=\id$ can be checked using the property $h[\sigma_0](\Omega_0+\sigma_0)=\id$ which follows from $h(\Omega_0)=\id$.

We remarked before that $\Omega_0+\sigma_t$ intersect the zero section of $\mathcal{E}\to E$ exactly in the graph $S_{\mu_0}$
of $\mu_0$. Moreover
this intersection is transversal and this implies that $M_t|_{S_{\mu_0}} \in \Gamma(\End(\mathcal{E}|_{S_{\mu_0}}))$ is invertible for all
$t\in [0,1]$ because a section of $\mathcal{E}$ which intersects $S$ transversally is mapped by $h$ to an
endomorphism of $\mathcal{E}$ that it invertible over $S$. By continuity of $M_t$ and compactness of $[0,1]$ there
is an open neighbourhood $V$ of $S_{\mu_0}$ in $E$ such that $M_t|_{V}$ is in $\Gamma(GL(\mathcal{E}|_{V}))$ for all
$t\in [0,1]$.

Next we modify $M_t$ such that it becomes invertible on the complement of $V$. First define
\begin{align*}
X_t:=\left(\frac{d}{dt}M_t\right)\circ M_t^{-1}
\end{align*}
on $V\times [0,1]$. It satisfies
\begin{align*}
X_t(\Omega_0+\sigma_t)=\left(\frac{d}{dt}M_t\right)(\Omega_0+\sigma_0)
\end{align*}
there. Choose a fiber metric $g$ on $\mathcal{E}\to E$ and define a family of fiberwise linear endomorphisms $Y_t$ of $\mathcal{E}$
by
\begin{align*}
Y_t: \mathcal{E}_{e}\xrightarrow{P(g)}<(\Omega_0+\sigma_t)|_{e}>\to<\left(\frac{d}{dt}M_t\right)(\Omega_0+\sigma_0)|_{e}>
\hookrightarrow \mathcal{E}_{e}
\end{align*}
on the complement of $S_{\mu_0}$. Here $P(g)$ denotes the orthogonal projection with respect to the chosen fiber metric
on the subvector bundle spanned by $\Omega_0+\sigma_t$. The smooth one-parameter family $(Y_t)_{t\in [0,1]}$ also satisfies
\begin{align*}
Y_t(\Omega_0+\sigma_t)=\left(\frac{d}{dt}M_t\right)(\Omega_0+\sigma_0)
\end{align*}
It is possible to find an open neighbourhood $W$ of $S_{\mu_0}$ in $V$ such that its closure $\overline{W}$ is still contained
in $V$. Consequently $(V,E\setminus \overline{W})$ is an open cover of $E$ and hence there is a partition of
unity $(\rho_1,\rho_2)$ subordinated to it. We set
\begin{align*}
Z_t:=\rho_1X_t+\rho_2 Y_t
\end{align*}
which is defined on all of $E$ and satisfies
\begin{align*}
Z_t(\Omega_0+\sigma_t)=\left(\frac{d}{dt}M_t\right)(\Omega_0+\sigma_0)
\end{align*}
there. The ordinary differential equation
\begin{align*}
\frac{d}{dt}N_t=Z_t\circ N_t, \quad N_0=\id
\end{align*}
can be solved fiberwise and one obtains a smooth one-parameter family $\hat{N}\in \Gamma(GL_{+}(\mathcal{E})_{[0,1]})$.
Furthermore one verifies
\begin{align*}
\frac{d}{dt}(N_t(\Omega_0+\sigma_0))=Z_t(N_t(\Omega_0+\sigma_0)), \quad N_0(\Omega_0+\sigma_0)=\Omega_0+\sigma_0
\end{align*}
which is exactly the flow equation satisfied by $(\Omega_0+\sigma_t)=M_t(\Omega_0+\sigma_0)$ and consequently
\begin{align*}
N_t(\Omega_0+\sigma_0)=\Omega_0+\sigma_t
\end{align*}
holds for arbitrary $t\in [0,1]$.

The smooth one-parameter family $(Z_t)_{t\in [0,1]}$ can be interpreted as an element of
$\Gamma(\mathcal{E}_{[0,1]}\otimes \mathcal{E}^{*}_{[0,1]})=\widetilde{BFV}^{(1,1)}(E)$ and as such it acts
on $BFV(E)$ via $[Z_t,\cdot]_{BFV}$. By Lemma \ref{lemma:integrability1} this smooth
one-parameter family of derivations integrates to a unique smooth one-parameter family of inner automorphisms
$\hat{\Phi}$ of $(BFV(E),[\cdot,\cdot]_{BFV})$.

By construction $(\varphi_{t}^{-1}\cdot\gamma_t=:\delta_t)_{t\in [0,1]}$ and $(\Phi_t\cdot \gamma_0)_{t\in [0,1]}$
are two smooth one-parameter families of Maurer--Cartan elements of $BFV(E)$ whose components in
$\Gamma(\mathcal{E}_{[0,1]})$ coincide. Consequently the components of the two smooth one-parameter
families of Maurer--Cartan elements $(\gamma_t)_{t\in [0,1]}$
and $((\varphi_t\circ \Phi_t)\cdot \gamma_0)_{t\in [0,1]}$ in $\Gamma(\mathcal{E}_{[0,1]})$ coincide.
Since $\hat{\gamma}$ is a family of geometric Maurer--Cartan elements so is $(\hat{\varphi}\circ\hat{\Phi})\cdot\gamma_0$.
So we constructed a gauge homotopy $((\hat{\varphi}\circ\hat{\Phi})\cdot\gamma_0,\hat{\varphi}\circ\hat{\Phi})$
from $\gamma_0$ to $(\varphi_1\circ \Phi_1)\cdot \gamma_0$.

To finish the proof we have to show that $\alpha \sim_{G} \gamma_0$ and $\beta \sim_{G} (\varphi_1\circ \Phi_1)\cdot \gamma_0$ hold.
This follows from the fact that the images  of $\alpha$ and $\gamma_0$ under $L_{\geo}$ on the one hand
and of $\beta$ and $(\varphi_1\circ \Phi_1)\cdot \gamma_0$ on the other hand are equal. By Theorem \ref{thm:geometricMCelements}
there are elements of $\underline{\Inn}{\ge 1}(BFV(E))$ that relate $\alpha$ to $\gamma_0$ and
$\beta$ to $(\varphi_1\circ \Phi_1)\cdot \gamma_0$ respectively. Such smooth one-parameter families of inner automorphisms yield
appropriate
gauge homotopies.
\end{proof}

\subsection{A morphism of groupoids}\label{ss:morphism}

The aim of this subsection is to extend the map
\begin{align*}
L_{\geo}: \mathcal{D}_{\geo}(E,\Pi) \twoheadrightarrow \mathcal{C}(E,\Pi)
\end{align*}
in a natural way
to a morphism of groupoids
\begin{align*}
\mathcal{L}_{\geo}: \hat{\mathcal{D}}_{\geo}(E,\Pi) \to \hat{\mathcal{C}}(E,\Pi)
\end{align*}
from the geometric BFV-groupoid $\hat{\mathcal{D}}_{\geo}(E,\Pi)$ to the groupoid of coisotropic sections $\hat{\mathcal{C}}(E,\Pi)$
of $(E,\Pi)$.

Let $(\hat{\alpha},\hat{\psi})$ be a gauge homotopy from the geometric Maurer--Cartan element $\alpha$ to 
the geometric Maurer--Cartan element $\beta$. Then the map introduced in the proof of Theorem \ref{thm:equivalenceclasses}
\begin{align*}
\hat{L}_{\geo}(\hat{\alpha},\hat{\psi}):=(L_{\geo}(\hat{\alpha}),\underline{L}(\hat{\psi}))
\end{align*}
is a Hamiltonian homotopy from $L_{\geo}(\alpha)$ to $L_{\geo}(\beta)$. It is straightforward to verify
\begin{align*}
\hat{L}_{\geo}\left((\hat{\beta},\hat{\phi})\Box_{\rho}(\hat{\alpha},\hat{\psi})\right)=
\hat{L}_{\geo}(\hat{\beta},\hat{\phi})\Box_{\rho}\hat{L}_{\geo}(\hat{\alpha},\hat{\psi})
\end{align*}
whenever the composition is defined, i.e. whenever the data attached to the boundary components $\{1\}$ and $\{0\}$ respectively match.
Furthermore $\hat{L}_{\geo}$ maps $\id_{\alpha}$ to $\id_{L_{\geo}(\alpha)}$
and can be extented to a map from isotopies of gauge homotopies to isotopies of Hamiltonian homotopies.

\begin{Lemma}\label{lemma:morphismofgroupoids}
The maps $L_{\geo}$ and $\hat{L}_{\geo}$ induce a morphism of groupoids
\begin{align*}
\mathcal{L}_{\geo}: \hat{\mathcal{D}}_{\geo}(E,\Pi) \to \hat{\mathcal{C}}(E,\Pi)
\end{align*}
that extends $L_{\geo}: \mathcal{D}_{\geo}(E,\Pi) \to \mathcal{C}(E,\Pi)$.
\end{Lemma}

\begin{Lemma}\label{lemma:surjectivity}
The morphism of groupoids
\begin{align*}
\mathcal{L}_{\geo}: \hat{\mathcal{D}}_{\geo}(E,\Pi) \to \hat{\mathcal{C}}(E,\Pi)
\end{align*}
is surjective on objects and surjective on all homomorphism sets.
\end{Lemma}

\begin{proof}
The surjectivity on the level of objects is content of part (a) of Theorem \ref{thm:geometricMCelements}.

Let $(-\hat{\mu},\hat{\phi})$ be a Hamiltonian homotopy from $-\mu$ to $-\nu$. In the proof
of Theorem \ref{thm:equivalenceclasses} a gauge homotopy from some geometric Maurer--Cartan element
$\alpha$ with $L_{\geo}(\alpha)=-\mu$ to some other geometric Maurer--Cartan element $\beta$ with
$L_{\geo}(\beta)=-\nu$ was constructed. It is straightforward to check that the image of this gauge homotopy
under $\hat{L}_{\geo}$ equals $(-\hat{\mu},\hat{\phi})$.
\end{proof}

The {\em kernel} $\ker(\mathcal{L}_{\geo})$ of $\mathcal{L}_{\geo}: \hat{\mathcal{D}}_{\geo}(E,\Pi) \to \hat{\mathcal{C}}(E,\Pi)$
is the normal full subgroupoid of $\hat{\mathcal{D}}_{\geo}(E,\Pi)$ whose hom-sets are given
by homomorphisms of $\hat{\mathcal{D}}_{\geo}(E,\Pi)$ that get mapped to Hamiltonian homotopies
which are equivalent to ones of the form $\id_{\mu}$ for $\mu$ some coisotropic section under $\simeq_{H}$ by $\mathcal{L}_{\geo}$.
Clearly all pure morphisms (see Definition \ref{def:pureclasses} in Subsection \ref{ss:BFV-groupoid})
of $\hat{\mathcal{D}}_{\geo}(E,\Pi)$ lie in the kernel of $\mathcal{L}_{\geo}$.
Proposition \ref{prop:kernel} in the following Subsection asserts that this is in fact all of $\ker(\mathcal{L}_{\geo})$. 

\subsection{An isomorphism of groupoids}\label{ss:isomorphism}

\begin{Proposition}\label{prop:kernel}
The kernel $\ker(\mathcal{L}_{\geo})$ of $\mathcal{L}_{\geo}: \hat{\mathcal{D}}_{\geo}(E,\Pi) \to \hat{\mathcal{C}}(E,\Pi)$
is exactly given by the class of pure morphisms of $\hat{\mathcal{D}}_{\geo}(E,\Pi)$.
\end{Proposition}

\begin{proof}
We have to show that the following implication holds: given an arbitrary gauge homotopy
$(\hat{\alpha},\hat{\phi})$ such that $\mathcal{L}_{\geo}(\hat{\alpha},\hat{\phi})\simeq_{H}\id_{\mu}$ 
for some coisotropic section $\mu$ of $(E,\Pi)$, then there is a {\em pure} gauge homotopy $(\hat{\beta},\hat{\psi})$
such that $(\hat{\alpha},\hat{\phi})\simeq_{G}(\hat{\beta},\hat{\psi})$ holds.

Let $(\hat{\gamma},\hat{\varphi})$ be a Hamiltonian homotopy that is isotopic to $\id_{\gamma_0}$. Hence there is an
isotopy of Hamiltonian homotopies $(\hat{\Gamma},\hat{\Phi})$ which starts at $(\hat{\gamma},\hat{\varphi})$
and ends at $\id_{\gamma_0}$. Moreover let $(\hat{\alpha},\hat{\phi})$ be a gauge homotopy
with $\mathcal{L}_{\geo}(\hat{\alpha},\hat{\phi})=(\hat{\gamma},\hat{\varphi})$.

In particular $\hat{\Gamma}$ is a section of $E_{[0,1]^{2}}$
such that $-\Gamma_{t,s}$ is a coisotropic section for arbitrary $t,s \in [0,1]$. Using the same lifting procedure
as described in the proof of Theorem \ref{thm:equivalenceclasses}, we obtain 
\begin{align*}
\hat{\Theta} \in \Gamma(\wedge \mathcal{E}_{[0,1]^{2}}\otimes \wedge \mathcal{E}^{*}_{[0,1]^{2}})
\end{align*}
such that $\Theta_{t,s}$ is a normalized Maurer--Cartan element satisfying $L_{\geo}(\Theta_{t,s})=\Gamma_{t,s}$
for all $t,s\in [0,1]$. Observe that $\Theta_{0,s}$ is constant in $s\in [0,1]$ and so is $\Theta_{1,s}$.
Furthermore lift $\hat{\Phi}$ to a smooth two-parameter family $\hat{\Psi}=\underline{R}(\hat{\Phi})$
of inner automorphism of $BFV(E)$.

Since the images of $\Theta_{0,0}$ and $\alpha_0$ under $L_{\geo}$ coincide, there is an element $\eta$ in $\Inn_{\ge 1}(BFV(E))$ such that
\begin{align*}
\alpha_0=\eta \cdot \Theta_{0,0}
\end{align*}
holds according to Theorem \ref{thm:geometricMCelements}.

Consider $\Sigma_{t,s}:=(\Psi_{t,s}^{-1}\circ \eta)\cdot \Theta_{t,s}$ which defines a smooth two-parameter family $\hat{\Sigma}$
of Maurer--Cartan elements. This family is in general not a family of geometric Maurer--Cartan elements.
However by definition $\Sigma_{0,s}=\alpha_0$ holds. We fix $s\in [0,1]$ and consider the smooth one-parameter family
of Maurer--Cartan elements $(\Sigma_{t,s})_{t\in [0,1]}$. Applying the gauging-procedure used in the proof of
Theorem \ref{thm:equivalenceclasses} one finds a smooth $(\Upsilon_{t,s})_{t\in [0,1]}\in \underline{\Inn}_{\ge 1}(BFV(E))$
such that the components of $\Upsilon_{t,s}\cdot \alpha_{0}$ and $\Sigma_{t,s}$ in $\Gamma(\mathcal{E}_{[0,1]})$ coincide.
Inspecting the construction of $\Upsilon_{t,s}$ reveals that it can be arranged 
such that the result is smooth with respect to $s\in [0,1]$ too. This is due to the  fact that $\Upsilon_{t,s}$
is constructed as the solution of some ordinary differential equation and as such depends smoothly on the input-data.

Because the components of $\Upsilon_{t,s}\cdot \alpha_{0}$ and $\Sigma_{t,s}=(\Psi_{t,s}^{-1}\circ \eta)\cdot \Theta_{t,s}$
in $\Gamma(\mathcal{E}_{[0,1]})$ coincide, so do the components of 
\begin{align*}
\Xi_{t,s}:=(\Psi_{t,s}\circ \Upsilon_{t,s})\cdot \alpha_0
\end{align*}
and $\eta\cdot \Theta_{t,s}$. The smooth two-parameter family of Maurer--Cartan elements
$\eta\cdot \Theta_{t,s}$ is geometric by construction, and consequently so is $\Xi_{t,s}$.

We constructed a smooth one-parameter family of inner automorphisms $\omega_{t,s}:=\Psi_{t,s}\circ \Upsilon_{t,s}$
and a smooth two-parameter family of geometric Maurer--Cartan elements $\Xi_{t,s}$ such that
\begin{align*}
\Xi_{t,s}=\omega_{t,s}\cdot \Xi_{0,s}
\end{align*}
holds for arbitrary $t,s\in [0,1]$. However this does not yield an isotopy of gauge homotopies since
$\Xi_{1,s}=(\Psi_{1,s}\circ \Upsilon_{1,s})\cdot \alpha_0$ which is not constant in $s\in [0,1]$.

By construction the component of the family
\begin{align*}
\Xi_{1,s}=\omega_{1,s}\cdot \Xi_{0,s}
\end{align*}
in $\Gamma(\mathcal{E}_{[0,1]})$ is equal to the component of $\eta\cdot \Theta_{1,s}$ which is constant in $s$.
So $\eta^{-1}\cdot \Theta_{1,s}$ is a family of normalized Maurer--Cartan elements with constant image under $L_{\nor}$.
For fixed $s\in [0,1]$ there is a smooth one-parameter family of automorphisms $\tau_{t,s}$ in $\Inn_{\ge 2}(BFV(E))$
such that $\tau_{1,s}\cdot (\eta^{-1}\cdot \Xi_{1,s})=\eta^{-1}\Xi_{1,0}$
due to Theorem \ref{thm:normalizedMCelements} in Subsection \ref{ss:normalizedMC-elements}. Inspecting the proof
of Theorem \ref{thm:normalizedMCelements} given in \cite{Schaetz} shows that $(\tau_{t,s})$ can be constructed such that
it is smooth also with respect to the parameter $s$.
Consequently
\begin{align*}
((\eta\circ \tau_{t,s}\circ \eta^{-1} \circ \omega_{t,s})\cdot \Xi_{0,s},\eta\circ \tau_{t,s}\circ \eta^{-1} \circ \omega_{t,s})
\end{align*}
is an isotopy of gauge homotopies.

The final step is to observe that the images of $\hat{\phi}$ and of $(\eta\circ \tau_{t,0}\circ \eta^{-1} \circ \omega_{t,0})$
under $\underline{L}$ coincide. Lemma \ref{lemma:kernel0} implies that there is a unique $\hat{\zeta}\in \underline{\Inn}_{\ge 1}(BFV(E))$
such that
\begin{align*}
\zeta_t=\phi_t \circ (\eta\circ \tau_{t,0}\circ \eta^{-1} \circ \omega_{t,0})^{-1}
\end{align*}
holds for arbitrary $t\in [0,1]$. This implies
\begin{align*}
\phi_t = \zeta_t \circ \eta\circ \tau_{t,0}\circ \eta^{-1} \circ \omega_{t,0}
\end{align*}
and consequently
\begin{align*}
\left((\zeta_t\circ \eta\circ \tau_{t,0}\circ \eta^{-1} \circ \omega_{t,0})\cdot \alpha_0,
(\zeta_t \circ \eta\circ \tau_{t,0}\circ \eta^{-1} \circ \omega_{t,0})\right)
\end{align*}
is an isotopy of gauge homotopies from $(\hat{\alpha},\hat{\psi})$ to $(\zeta_t\circ \eta\circ \tau_{t,1}\circ \eta^{-1})$.
The latter gauge homotopy is pure.
\end{proof}

Proposition \ref{prop:kernel} implies that the morphism 
\begin{align*}
\mathcal{L}_{\geo}: \hat{\mathcal{D}}_{\geo}(E,\Pi) \twoheadrightarrow \hat{\mathcal{C}}(E,\Pi)
\end{align*}
factors through the natural projection $\hat{\mathcal{D}}_{\geo}(E,\Pi)\twoheadrightarrow \hat{\mathcal{D}}(E,\Pi)$ and induces
an isomorphism of groupoids
\begin{align*}
\mathcal{L}: \hat{\mathcal{D}}(E,\Pi) \xrightarrow{\cong} \hat{\mathcal{C}}(E,\Pi).
\end{align*}
We conclude with the following ``categorification'' of Theorem \ref{thm:equivalenceclasses} Subsection \ref{ss:isomorphism_moduli}

\begin{Theorem}\label{thm:main}
Let $(E,\Pi)$ be a coisotropic vector bundle. 
Then the morphism of groupoids
\begin{align*}
\mathcal{L}_{\geo}: \hat{\mathcal{D}}_{\geo}(E,\Pi) \twoheadrightarrow \hat{\mathcal{C}}(E,\Pi)
\end{align*}
introduced in Lemma \ref{lemma:morphismofgroupoids}
induces an isomorphism
\begin{align*}
\mathcal{L}: \hat{\mathcal{D}}(E,\Pi) \xrightarrow{\cong} \hat{\mathcal{C}}(E,\Pi)
\end{align*}
between the BFV-groupoid $\hat{\mathcal{D}}(E,\Pi)$ (Definition \ref{def:BFV-groupoid} in Subsection \ref{ss:BFV-groupoid}) and
the groupoid of coisotropic sections $\hat{\mathcal{C}}(E,\Pi)$
(Definition \ref{def:groupoidofcoisotropicsections} in Subsection \ref{ss:groupoidofcoisotropicsections}).
\end{Theorem}

\end{document}